\documentclass[11pt]{amsart}
\usepackage[all]{xy}

\usepackage[utf8]{inputenc}		
\usepackage[T1]{fontenc}
\usepackage[english]{babel}

\usepackage{amsfonts}			
\usepackage{amsmath}
\usepackage{amssymb}
\usepackage{amsthm}
\usepackage{mathtools}
\makeatletter					
\@namedef{subjclassname@2020}{	
	\textup{2020} Mathematics Subject Classification}
\makeatother

\usepackage{graphicx}			
\usepackage{tikz}
	\usetikzlibrary{cd}

\usepackage{hyperref}			
	\hypersetup{colorlinks=false, urlcolor=black, linkcolor=black}

\usepackage{enumitem}			

\usepackage{color}				

\newcommand{\Z}{\mathbb{Z}}						
\newcommand{\R}{\mathbb{R}}						
\newcommand{\C}{\mathbb{C}}						

\renewcommand{\S}{\mathbb{S}}					
\newcommand{\D}{\mathbb{D}}						
\newcommand{\B}{\mathbb{B}}

\newcommand{\eps}{\varepsilon}					

\newcommand{\dd}								
	{\mathop{}\!\mathrm{d}}						
\newcommand{\ddn}[1]							
	{\mathop{}\!\mathrm{d^{#1}}}

\newcommand{\abs}[1]							
	{\left| #1 \right|}
\newcommand{\smallabs}[1]						
	{\lvert #1 \rvert}	
\newcommand{\norm}[1]							
	{\left\lVert #1 \right\rVert}	
\newcommand{\smallnorm}[1]						
	{\lVert #1 \rVert}						
\newcommand{\ip}[2]								
	{\left< #1 , #2 \right>}

\DeclareMathOperator{\vol}{vol}					

\DeclareMathOperator{\dist}{dist}					

\let\Re\relax									
\let\Im\relax
\DeclareMathOperator{\Re}{Re}
\DeclareMathOperator{\Im}{Im}					

\newcommand{\loc}{\mathrm{loc}}					

\newcommand{\cM}{\mathcal{M}}

\def\Xint#1{\mathchoice
	{\XXint\displaystyle\textstyle{#1}}%
	{\XXint\textstyle\scriptstyle{#1}}%
	{\XXint\scriptstyle\scriptscriptstyle{#1}}%
	{\XXint\scriptscriptstyle\scriptscriptstyle{#1}}%
	\!\int}
\def\XXint#1#2#3{{\setbox0=\hbox{$#1{#2#3}{\int}$}
		\vcenter{\hbox{$#2#3$}}\kern-.5\wd0}}

\def\dashint{\Xint-}

\newtheorem{thm}{Theorem}[section]{\bf}{\it}
\newtheorem{lemma}[thm]{Lemma}

\newtheorem{prob}[thm]{Problem}
\newtheorem{conj}[thm]{Conjecture}

\newtheorem{innercustomthm}{Theorem}[section]{\bf}{\it}
\newenvironment{customthm}[1]
	{\renewcommand\theinnercustomthm{#1}\innercustomthm}
	{\endinnercustomthm}
	
{\bf}{\it}

\theoremstyle{definition}

\theoremstyle{remark}

\numberwithin{equation}{section}

\begin{document}
	
\title{Mappings of generalized finite distortion  and  continuity}

\author{Anna Dole\v zalov\'a} 
\address{Department of Mathematical Analysis, Charles University,  Sokolovsk\'a 83, 186 00 Prague 8, Czech Republic}
\email{dolezalova@karlin.mff.cuni.cz}

\author{Ilmari Kangasniemi}
\address{Department of Mathematics, Syracuse University, Syracuse,
NY 13244, USA }
\email{kikangas@syr.edu}

\author{Jani Onninen}
\address{Department of Mathematics, Syracuse University, Syracuse,
NY 13244, USA and  Department of Mathematics and Statistics, P.O.Box 35 (MaD) FI-40014 University of Jyv\"askyl\"a, Finland
}
\email{jkonnine@syr.edu}
	
\thanks{A.\ Dole\v zalov\'a is a Ph.D.\ student in the University Centre for Mathematical Modelling, Applied Analysis and Computational Mathematics (Math MAC) and was supported by the grant GA CR P201/21-01976S and by the project Grant Schemes at CU, reg.\ no.\ CZ.02.2.69/0.0/0.0/19 073/0016935. J.\ Onninen was supported by the NSF grant DMS-2154943.}
\subjclass[2020]{Primary 30C65; Secondary 35R45}
\keywords{Mappings of finite distortion, Quasiregular mappings,  continuity, differential inclusions, quasiregular values}
	
\begin{abstract}
	We study continuity properties of Sobolev mappings $f \in W_{\loc}^{1,n} (\Omega, \R^n)$, $n \ge 2$, that satisfy the following generalized finite distortion inequality
	\[\abs{Df(x)}^n \leq K(x) J_f(x) + \Sigma (x)\] 
	for almost every $x \in \R^n$. Here $K \colon \Omega \to  [1, \infty)$ and  $\Sigma \colon \Omega \to  [0, \infty)$ are measurable functions.  Note that when $\Sigma \equiv 0$, we recover the class of mappings of finite distortion, which are always continuous. The continuity of arbitrary solutions, however, turns out to be an intricate question. We fully solve the continuity problem in the case of bounded distortion $K \in L^\infty (\Omega)$, where a sharp condition for continuity is that $\Sigma$ is in the Zygmund space $\Sigma \log^\mu(e + \Sigma) \in L^1_\loc(\Omega)$ for some $\mu > n-1$. We also show that one can slightly relax the boundedness assumption on $K$ to an exponential class $\exp(\lambda K) \in L^1_\loc(\Omega)$ with $\lambda > n+1$, and still obtain continuous solutions when $\Sigma \log^\mu(e + \Sigma) \in L^1_\loc(\Omega)$ with $\mu > \lambda$. On the other hand, for all $p, q \in [1, \infty]$ with $p^{-1} + q^{-1} = 1$, we construct a discontinuous solution with $K \in L^p_\loc(\Omega)$ and $\Sigma/K \in L^q_\loc(\Omega)$, including an example with $\Sigma \in L^\infty_\loc(\Omega)$ and $K \in L^1_\loc(\Omega)$.
\end{abstract}
	
\maketitle
	
\section{Introduction}
Let $\Omega$ be a connected, open subset of $\R^n$ with $n \ge 2$. Recall that a \emph{differential inclusion} is a condition requiring that, for almost every (a.e.) $x \in \Omega$, a weakly differentiable mapping  $f \in W^{1,1}_\loc(\Omega, \R^m)$ satisfies $Df(x) \in F(x, f(x))$ where $F$ is a function from $\Omega \times \R^m$ to subsets of $m \times n$-matrices. Here, we are searching for differential inclusions under which a Sobolev map $f\in W_{\loc}^{1,n} (\Omega, \R^n)$  has a  continuous representative. More specifically, we are interested in  ones which are motivated by the Geometric Function Theory, with connections to mathematical models of Nonlinear Elasticity. 

This leads us to consider the differential inclusions given by the set functions
\begin{equation}
\cM_n (K, \Sigma) \colon x \mapsto \{ A \in \mathbb R^{n \times n} \colon \abs{A}^n \leq K(x) \det A + \Sigma(x)\} \,  ,
\end{equation}
where  $K\colon \Omega \to [1, \infty)$ and $\Sigma \colon \Omega \to [0, \infty)$ are  given measurable functions. Here and in what follows, $\abs{A}$ stands for  the operator norm of matrix $A\in \mathbb R^{n\times n}$;  that is, $\abs{A} = \sup \{\abs{A h} : h \in \S^{n-1} \}$. We also use the shorthand $G \in \cM_n(K, \Sigma)$ if $G \colon \Omega \to \R^{n \times n}$ satisfies $G(x) \in \cM_n(K, \Sigma)(x)$ for a.e.\ $x \in \Omega$. Now, our continuity problem reads as follows.

\begin{prob}\label{prob:continuity problem}
	Find a necessary and sufficient condition on the functions $K$ and $\Sigma$ which guarantees that if $f\in W_{\loc}^{1,n} (\Omega, \R^n)$ with $Df \in \mathcal M_n (K, \Sigma)$, then $f$ has a continuous representative.
\end{prob}

A necessary condition for Problem \ref{prob:continuity problem} is that $\Sigma$ must at least to lie in the Zygmund space $L\log^\mu L_{\loc} (\Omega)$ for some $\mu > n-1$:  that is, 
\begin{equation}\label{eq:min_assumption_sigma}
	\Sigma \log^\mu (e+ \Sigma) \in L^1_{\loc} (\Omega) \qquad \mu > n-1 \, . 
\end{equation} 
Indeed, the mapping $f \colon \B^n(0, 1) \to \R^n$ defined by 
\begin{equation}\label{eq:log_log_log}
	f(x) = \left(\log \log \log \frac{e^e}{\abs{x}}, 0, \dots , 0\right)
\end{equation}
has $ \det Df \equiv 0$ and $\abs{Df}^n \log^{n-1}(e + \abs{Df}^n) \in L^1 (\B^n(0, 1))$, but $\lim\limits_{x \to 0} \abs{f(x)} = \infty$. 

\subsection{Results for bounded $K$.}

When $\Sigma \equiv 0$ and $K\in L^\infty(\Omega)$, $\cM_n(K, 0)$ recovers the \emph{mappings of bounded distortion}, also
known as \emph{quasiregular} mappings; a mapping  $f\colon \Omega \to \R^n$ is \emph{$K$-quasiregular} for $K \in [1, \infty)$ if $f\in W_{\loc}^{1,n} (\Omega, \R^n)$ with $\abs{Df(x)}^n \leq K \det Df(x)$ for a.e.\ $x \in \Omega$. Homeomorphic $K$-quasiregular mappings are  called \emph{$K$-quasiconformal}. The first breakthrough in the theory of mappings of bounded distortion was Reshetnyak's theorem on  H\"older continuity: a $K$-quasiregular mapping is locally $1/K$-H\"older continuous, see~\cite{Reshetnyak_continuity} and~\cite[Corollary II.1]{Reshetnyak-book}. Such H\"older continuity properties of quasiconformal mappings in the plane were earlier established by Morrey~\cite{Morrey}.

Other differential inclusions of the type $\cM_n(K, \Sigma)$ with $K \in L^\infty(\Omega)$ have also arisen naturally in different contexts. For instance, Simon~\cite{Simon_Holder} developed a local regularity theory for minimal graphs of functions $u \colon \R^2 \to \R$ such that the Gauss map of the graph of $u$ satisfies
\begin{equation}\label{eq:qr_gene1def}
	\abs{Df(x)}^2 \leq K \det Df(x) + \Sigma
\end{equation}
where $1\le K <\infty$ and $0\le \Sigma < \infty$ are given constants. Recall that the \emph{Gauss map} takes the points of a surface $S \subset \R^n$ to the unit normal vector in $\S^{n-1}$. In particular, the Gauss map automatically satisfies \eqref{eq:qr_gene1def} when $u$ is a solution of any equation of
mean curvature type~\cite[(1.9) (ii)]{Simon_Holder}.  Similar results for simply connected surfaces embedded in $\R^3$ are due to Schoen and Simon~\cite{Schoen-Simon}. The main result in~\cite{Simon_Holder} enabling the regularity theory states that a local $W^{1,2}$-solution to~\eqref{eq:qr_gene1def} between embedded 2D-surfaces is H\"older continuous; see also~\cite[Ch. 12]{Gilbarg-Trudinger_Book}. 

This H\"older continuity result has been generalized for unbounded $\Sigma$ as well. Precisely, if $K\in L^\infty$ and $\Sigma \in L^{p}_{\loc} (\Omega)$ for some $p>1$, then a mapping $f\in W_{\loc}^{1,n} (\Omega, \R^n)$ with $Df\in \mathcal M_n (K,\Sigma)$ has a H\"older continuous representative. For the planar case, see the proof of \cite[Theorem 8.5.1]{Astala-Iwaniec-Martin_Book} by Astala, Iwaniec and Martin, and for the more general case $n \geq 2$, see the argument in \cite[Section 3]{Kangasniemi-Onninen_Heterogeneous} by Kangasniemi and Onninen. While the planar argument of Astala, Iwaniec and Martin relies on complex potential theory, the higher dimensional proof is closer to that of Simon \cite{Simon_Holder}, mimicking the lines of reasoning by Morrey~\cite{Morrey} and Reshetnyak~\cite{Reshetnyak_continuity} in the case of mappings of bounded distortion.

However, despite yielding sharp results on the $L^p$-scale, the Morrey-type decay argument used in \cite[Section 3]{Kangasniemi-Onninen_Heterogeneous} does not give a sharp result if one moves to the Zygmund space setting $\Sigma \in L\log^\mu L_{\loc} (\Omega)$. In particular, the decay argument shows continuity when $\mu > n$, but the optimal regularity assumption for $\Sigma$ is in fact $\mu > n-1$, precisely the minimal necessary condition stated in \eqref{eq:min_assumption_sigma}. This optimal regularity theorem is our first main result.

\begin{thm}\label{thm:bounded_K}
	Suppose that $f \in W_{\loc}^{1,n} (\Omega, \R^n)$ and $Df(x) \in \mathcal M_n (K, \Sigma)(x)$ a.e.\ in $\Omega$, with
	\begin{align*}
		K \in L_{\loc}^\infty (\Omega) \qquad \text{and} \qquad \Sigma \log^{\mu} \left( e + \Sigma \right) \in L^1_\loc(\Omega),
	\end{align*}
	for some $\mu > n-1$. Then $f$ has a continuous representative. 
\end{thm}

Furthermore, under the assumptions of Theorem \ref{thm:bounded_K}, the local modulus of continuity
\begin{equation}\label{eq:mod_cont}
	 \omega_f (x_0,r) = \sup  \{ \abs{f(x_0)-f(x)}  \colon x\in \Omega \, ,  \; \;  \abs{x-x_0} \le r \}
\end{equation}
is majorized by $C \log^{-(\mu-n+1)/n} (1/{r})$ for $x_0 \in \Omega$ and small $r>0$. 	
By considering functions of the form $f(x) = (\log^{-\alpha} \abs{x}^{-1}, 0, \dots, 0)$ with $\alpha > 0$, it is easy to see that the above exponent $(\mu-n+1)/n$ is sharp.

Theorem~\ref{thm:bounded_K} is obtained by proving the following sharp higher integrability result for $Df$ on the Zygmund scale. 

\begin{thm}\label{thm:bounded_K_higher_int}
	Suppose that $f \in W_{\loc}^{1,n} (\Omega, \R^n)$ and $Df \in \mathcal M_n(K, \Sigma)$ with
	\begin{align*}
		K \in L_{\loc}^\infty (\Omega) \qquad \text{and} \qquad \Sigma \log^{\mu} \left( e + \Sigma \right) \in L^1_\loc(\Omega),
	\end{align*}
	for some $\mu \ge 0$. Then $\abs{Df}^n \log^\mu (e+\abs{Df}) \in  L^1_{\loc} (\Omega)$. 
\end{thm}

It is worth noting that  the sharp local $1/K$-H\"older continuity result for spatial $K$-quasiregular mappings cannot be obtained from known higher integrability results. Indeed, while $K$-quasiregular mappings have been shown to belong to the Sobolev space $W_{\loc}^{1,pn}(\Omega, \R^n)$ for some $p>1$~\cite{Gehring,  Meyers-Elcrat_HigherInt}, the sharp exponent $p=p(n,K)$ remains unknown when $n\ge 3$.  A well-known conjecture asserts that
\begin{equation}\label{eq:sharp_K} p(n,K)= \frac{K}{K-1} \, . \end{equation}
In a seminal work, Astala~\cite{Astala-ACTA} established the sharp exponent in the planar case.

This conjecture also has a counterpart for mappings $f\in W_{\loc}^{1,n} (\Omega, \R^n)$ with $Df \in \mathcal M_n (K, \Sigma)$. Indeed, if $\norm{K}_{L^\infty(\Omega)} \leq K_\circ$, we expect that $f \in W_{\loc}^{1,pn}(\Omega, \R^n)$ whenever $\Sigma \in L^p_{\loc}(\Omega)$ for all $p \leq p(n, K_\circ)$, where $p(n, K_\circ)$ is as in \eqref{eq:sharp_K}. This is the maximal amount of higher integrability of $Df$ possible when $\Sigma \in L^p_{\loc}(\Omega)$, which can be seen by taking $f = (g, 0, \dots, 0)$ and $\Sigma = \abs{\nabla g}^n$, where $g$ is any function in $W^{1, pn}_\loc(\Omega) \setminus \bigcup_{q > n} W^{1, qn}_\loc(\Omega)$. However, similar to the quasiregular theory, current tools are only enough to prove a result like this with an unknown value of $p(n, K_\circ)$. 

\begin{thm}\label{thm:bounded_K_higher_int_Lp}
	For given $n\ge 2$ and $K_\circ \in [1, \infty)$, there exists a value $p(n, K_{\circ}) > 1$, such that if $f \in W_{\loc}^{1,n} (\Omega, \R^n)$ and $Df \in \mathcal M_n (K, \Sigma)$ with
	\begin{align*}
		\norm{K}_{L^\infty(\Omega)} \leq K_\circ \qquad \text{and} \qquad \Sigma \in L^{p}_\loc(\Omega),
	\end{align*}
	for some $p \in [1, p(n, K_{\circ}))$, then $\abs{Df}^n  \in  L^{p}_{\loc} (\Omega)$. 
\end{thm}

\subsection{Results for general $K$}

In the last 20 years, systematic studies of \emph{mappings of finite distortion} have emerged in the field of  geometric function theory. Recall that a mapping $f\in W^{1,n}_{\loc} (\Omega, \R^n)$ has finite distortion if $\abs{Df(x)}^n \leq K(x) \det Df(x)$ a.e.\ on $\Omega$ for some measurable $K \colon \Omega \to [1, \infty)$: that is, if $Df\in \mathcal M_n(K,0)$.  Thus,  the class of mappings of finite distortion extends the theory of mappings of bounded distortion to the degenerate elliptic setting, \cite{Iwaniec-Martin_book, Hencl-Koskela-book}.  There one finds applications in materials science, particularly in nonlinear elasticity. The mathematical models of nonlinear elasticity have been pioneered by Antman~\cite{Antman_Elasticity-book}, Ball~\cite{Ball_nonlinear-elasticity} and Ciarlet~\cite{Ciarlet_Elasticity-book}.
 
In general, some bounds on the distortion are needed to obtain a full theory, analogous to the theory of quasiregular maps. The continuity property, however, follows without any restriction on the distortion function $K$. Precisely, if $K\colon \Omega \to [1,\infty )$ is any measurable function, then a Sobolev mapping $f\in W^{1,n}_{\loc} (\Omega, \R^n)$ with $Df\in \mathcal M_n (K,0)$ has a continuous representative~\cite{Godshtein-Vodopyanov_Continuity, Iwaniec-Koskela-Onninen-Invent}. 

Surprisingly, the continuity problem becomes a lot more challenging when $\Sigma \not\equiv 0$. Our next result shows that the solutions need not be continuous even in the case of bounded $\Sigma$ if the distortion $K$ is just a measurable function.

\begin{thm}\label{thm:discontinuity_bounded_sigma} 
	There exist a domain $\Omega \subset \mathbb R^2$ and a Sobolev map $f\in W^{1,2} (\Omega , \mathbb R^2)$ such that $0\in \Omega$, $f \in C (\Omega \setminus \{0\}, \mathbb R^2)$, $\lim_{x \to \infty} \abs{f(x)}=\infty$, and $Df \in \mathcal M_2(K, \Sigma)$ with
 	\[
 		\Sigma \in L^\infty (\Omega ) 
 		\qquad \text{and} \qquad K\in L^{1}(\Omega)\,.  
 	\]
\end{thm}

On the other hand, it is well known that mappings of exponentially integrable distortion behave in many ways like quasiregular mappings~\cite{Hencl-Koskela-book}. For instance, if a nonconstant Sobolev mapping $f \colon \Omega \to \R^n$ satisfies $Df\in \mathcal M_n (K,0)$ with $\exp (\lambda K) \in L^{1}(\Omega)$ and $\lambda > 0$, then $f$ is both discrete and open~\cite{Kauhanen-Koskela-Maly}. Moreover, the local modulus of continuity $\omega_f(x_0 , r)$ of $f$ is majorized up to a multiplicative constant by $\log^{-\lambda/n} (1/r)$ if $x_0 \in \Omega$ and $r>0$ is sufficiently small~\cite{Iwaniec-Koskela-Onninen-Invent}. This raises a natural question in the general case $Df\in \mathcal M_n (K,\Sigma)$: is there a version of the continuity result of Theorem~\ref{thm:bounded_K} where the boundedness assumption $K \in L^\infty_\loc(\Omega)$ has been relaxed to $\exp(\lambda K) \in L^1_\loc(\Omega)$ for some $\lambda > 0$.
The next result shows that this is not the case for arbitrary $\lambda > 0$.  
\begin{thm}\label{thm:discontinuity_K_Lp}  
	For every $\mu \in (0, 2)$, there exist a domain $\Omega \subset \mathbb R^2$ and a Sobolev map $f\in W^{1,2} (\Omega, \mathbb R^2)$ such that $0\in \Omega$, $f \in C (\Omega \setminus \{0\}, \mathbb R^2)$, $\lim_{x \to \infty} \abs{f(x)}=\infty$, and $Df \in \mathcal M_2(K, \Sigma)$ with
	\[  
		\exp (\lambda K )\in L^{1}(\Omega)  
		\qquad \text{and} \qquad   
		\Sigma \log^\mu (e+\Sigma) \in L^1(\Omega) \, 
	\]
	for every $\lambda > 0$.
 \end{thm}
Nevertheless, it is possible to obtain a modulus of continuity in the case with $\exp(\lambda K) \in L^1_\loc(\Omega)$ and $\Sigma \log^\mu (e+\Sigma) \in L^1_\loc(\Omega)$, if one assumes $\lambda$ and $\mu$ to be sufficiently large.

\begin{thm}\label{thm:exp_cont_modulus}
	Let $\Omega \subset \R^n$ be a  domain, and let $f \in W^{1,n}_\loc(\Omega, \R^n)$ and   $Df \in \mathcal M_n (K, \Sigma)$ with 
	\begin{align*}
		\exp (\lambda K) \in L^1_\loc(\Omega) \qquad \text{and} \qquad \Sigma \log^{\mu} \left( e + \Sigma \right) \in L^1_\loc(\Omega),
	\end{align*}
	for some $\mu > \lambda > n+1$. Then $f$ has a continuous representative. 
	
	In particular, for all $x_0 \in \Omega$ and sufficiently small $r>0$, we have the following local modulus of continuity estimate:
	\[ 
		\omega_f (x_0, r) \le C \log^{-\alpha} (1/r)  \quad \textnormal{where } \alpha =\frac{\lambda-n-1}{n}\, . 
	\]
\end{thm}

\subsection{Single-value theory}

Understanding the pointwise behavior of quasi\-regular mappings motivates us to study a variant of the differential inclusion of $\cM_n(K, \Sigma)$. In particular, given $K, \Sigma \colon \Omega \to \R^n$ and $y_0 \in \R^n$, we define a map $\cM_n(K, \Sigma, y_0)$ from $\Omega \times \R^n$ to subsets of $\R^{n \times n}$ by 
\begin{multline}\label{eq:MFDvalue}
	\cM_n(K, \Sigma, y_0) \colon \\
	(x,y) \mapsto \{ A \in \mathbb R^{n \times n} \colon \abs{A}^n \leq K(x) \det A + \abs{y - y_0}^n \Sigma(x)\}.
\end{multline}
Consequently, we obtain a differential inclusion by requiring that $Df(x) \in \cM_n(K, \Sigma, y_0)(x, f(x))$ for a.e.\ $x \in \Omega$, which we again denote by the shorthand $Df \in \cM_n(K, \Sigma, y_0)$.

For $K \in L^\infty(\Omega)$, the differential inclusion $Df \in \cM_n(K, \Sigma, y_0)$ leads to the theory of \emph{quasiregular values} developed by the last two authors in \cite{Kangasniemi-Onninen_Heterogeneous} and \cite{Kangasniemi-Onninen_1ptReshetnyak}. This term is motivated by the fact that for bounded $K \in L^\infty(\Omega)$, solutions of $Df \in \cM_n(K, \Sigma, y_0)$ satisfy a single-value version of the celebrated Reshetnyak's theorem at $y_0$. Precisely, if $f \in W^{1,n}_\loc(\Omega, \R^n)$ is non-constant and $Df \in \cM_n(K, \Sigma, y_0)$ with $K \in [1, \infty)$ constant and $\Sigma \in L^p_\loc(\Omega)$ for some $p > 1$, then $f$ is continuous, $f^{-1}\{y_0\}$ is discrete, the local index $i(x, f)$ is positive in $f^{-1}\{y_0\}$, and every neighborhood of a point of $f^{-1}\{y_0\}$ is mapped to a neighborhood of $y_0$: see \cite[Theorem 1.2]{Kangasniemi-Onninen_1ptReshetnyak}.

Notably, the additional term $\abs{f-y_0}^n$ in the differential inclusion $Df \in \cM_n(K, \Sigma, y_0)$ causes no additional difficulty in our continuity problem on the $L^p$-scale. Indeed, if $f \in W_{\loc}^{1,n} (\Omega, \R^n)$ and $Df \in \cM_n(K, \Sigma, y_0)$ with $\Sigma \in L^p_\loc(\Omega)$, $p > 1$, one can define $\Sigma_0 = \abs{f - y_0}^n \Sigma$ and conclude using the Sobolev embedding theorem that $\Sigma_0 \in L^q_\loc(\Omega)$ for every $q \in [1, p)$. The question then reduces to the continuity of solutions of $Df \in \cM_n(K, \Sigma_0)$. 

The sharpness of such an approach, however,  becomes an issue when one moves to the Zygmund space scale of \eqref{eq:min_assumption_sigma}. Indeed, if $f \in W_{\loc}^{1,n} (\Omega, \R^n)$ satisfies $Df \in \cM_n(K, \Sigma, y_0)$ with $\Sigma \log^\mu (e + \Sigma) \in L^1_{\loc} (\Omega)$, then it can be shown using the Moser-Trudinger inequality that $\Sigma_0 = \abs{f - y_0}^n \Sigma$ satisfies $\Sigma_0 \log^{\mu-n+1} (e + \Sigma_0) \in L^1_{\loc} (\Omega)$. Theorem~\ref{thm:bounded_K} hence yields that $f$ has a continuous representative if $\mu-n+1 > n-1$, i.e.\ $\mu > 2n - 2$. 

This result for $\mu > 2n - 2$, however, turns out to be far from optimal. This is because, by an iteration argument using Theorem \ref{thm:bounded_K_higher_int}, this gap from \eqref{eq:min_assumption_sigma} can be entirely eliminated. Again the mapping $f(x) = (\log \log \log (e^e / \abs{x}), 0, \dots , 0)$ on $\B^n(0, 1)$ shows that the following theorem is sharp.
	
\begin{thm}\label{thm:K_bounded_y_0}
	Let $\Omega$ be a domain in $\R^n$. Suppose that a Sobolev mapping $f \in W_{\loc}^{1,n} (\Omega, \R^n)$ satisfies $Df \in \cM_n(K, \Sigma, y_0)$ with $K \colon \Omega \to [1, \infty)$, $\Sigma \colon \Omega \to [0, \infty)$ and $y_0 \in \R^n$.  If 
	\begin{align*}
		K \in L_{\loc}^\infty (\Omega) \qquad \text{and} \qquad \Sigma \log^{\mu} \left( e + \Sigma \right) \in L^1_\loc(\Omega),
	\end{align*}
	for some $\mu > n-1$, then $f$ has a continuous representative. 
\end{thm}

However, in the case of $Df \in \cM_n(K, \Sigma, y_0)$ with exponentially integrable $K$, the use of this trick is prevented as our results are not based on higher integrability. Hence, the current best bound in this case is the following result, given by the above Moser-Trudinger -argument combined with Theorem~\ref{thm:exp_cont_modulus}.

\begin{thm}\label{thm:MFDvalue_cont}
	Let $\Omega$ be a domain in $\R^n$. Suppose that a Sobolev mapping $f \in W_{\loc}^{1,n} (\Omega, \R^n)$ satisfies $Df \in \cM_n(K, \Sigma, y_0)$ with $K \colon \Omega \to [1, \infty)$, $\Sigma \colon \Omega \to [0, \infty)$ and $y_0 \in \R^n$. If
	\begin{align*}
		\exp (\lambda K) \in L^1_\loc(\Omega) \qquad \text{and} \qquad \Sigma \log^{\mu} \left( e + \Sigma \right) \in L^1_\loc(\Omega),
	\end{align*}
	for some $\mu > \lambda + n - 1 > 2n$, then $f$ has a continuous representative. 
	
	In particular, for all $x_0 \in \Omega$ and sufficiently small $r>0$, we have the following local modulus of continuity estimate:
	\[ 
		\omega_f (x_0, r) \le C \log^{-\alpha} (1/r)  \quad \textnormal{where } \alpha =\frac{\lambda-n-1}{n}\, . 
	\]
\end{thm}

\subsection{$L^p$-integrable $K$}

In the case where $K \in L^p_\loc(\Omega)$ with $p \in [1, \infty]$, we conjecture that Problem \ref{prob:continuity problem} has a positive answer if $\Sigma \in L^q_\loc(\Omega)$ for any $q > p^*$, where $p^*$ is the H\"older conjugate of $p$. In fact, we  conjecture that a stronger statement is true, where $\Sigma \in L^q_\loc(\Omega)$ can be replaced by the hypothesis $\Sigma / K \in L^q_\loc(\Omega)$.

\begin{conj} \label{conj}  
	Let $1\le p, q \le \infty$. Suppose that $f \in W_{\loc}^{1,n} (\Omega, \R^n)$, $Df \in \mathcal M_n (K, \Sigma)$ with $K \geq 1$, $\Sigma \geq 0$, 
	\begin{align*}
			K \in L_{\loc}^p  (\Omega), \qquad \text{and} \qquad  \frac{\Sigma}{K} \in L^q_\loc(\Omega), \qquad \textnormal{where} \;\; \frac{1}{p} + \frac{1}{q} <1 \,. 
		\end{align*}
	Then $f$ has a continuous representative.
 \end{conj}

In order to justify the assumption $p^{-1} + q^{-1} < 1$ of Conjecture \ref{conj}, we point out that we have a discontinuous example in the case $p = 1, q = \infty$ due to Theorem~\ref{thm:discontinuity_bounded_sigma}. Moreover, in the case $q = 1, p = \infty$, the triple logarithm map \eqref{eq:log_log_log} provides a discontinuous example. The necessity of the assumption for the remaining cases $1 < p < \infty$ is then given by the following example. 

\begin{thm}\label{thm:cusp_ex_version_2}
	Let $p, q \in (1, \infty)$. If $p^{-1} + q^{-1} \geq 1$, then there exists a domain $\Omega \subset \mathbb R^2$ and a Sobolev map $f\in W^{1,2} (\Omega, \mathbb R^2)$ such that $0\in \Omega$, $f \in C (\Omega \setminus \{0\}, \mathbb R^2)$, $\lim_{x \to \infty} \abs{f(x)}=\infty$, and $Df \in \mathcal M_2(K, \Sigma)$ with
	\begin{equation}\label{eq:K_SigmaperK_duality}  
		K \in L^p  (\Omega) \qquad \text{and} \qquad \frac{\Sigma}{K} \in L^q (\Omega).
	\end{equation}
\end{thm}
Furthermore, we give several versions of Theorem~\ref{thm:cusp_ex_version_2} where \eqref{eq:K_SigmaperK_duality} is replaced by a condition of the type 
	\[  
		K \in L^p  (\Omega) \qquad \text{and} \qquad \Sigma \in L^s (\Omega),
	\]
see Theorems~\ref{lem:is_this_also_geq} and~\ref{thm:spiral_ex_version_2} for details.

\section{Results based on higher integrability}

In this section, we prove the continuity results that are based on higher integrability: Theorems \ref{thm:bounded_K}, \ref{thm:bounded_K_higher_int}, \ref{thm:bounded_K_higher_int_Lp}, and \ref{thm:K_bounded_y_0}.

\subsection{Higher integrability on the $L^p$-scale.}

The higher integrability result of Theorem \ref{thm:bounded_K_higher_int_Lp} is essentially the same as \cite[Lemma 6.1]{Kangasniemi-Onninen_1ptReshetnyak}, with only minor tweaks to account for the non-constant $K$. We regardless recall the argument for the convenience of the reader, as we require the reverse H\"older inequality proven during the argument for our later proof of Theorem \ref{thm:bounded_K_higher_int}.

If $B=\B^n(x, r)$ is a ball and $c \in (0, \infty)$, then we denote $cB = \B^n(x, cr)$. Similarly, if $Q = x + (-r, r)^n \subset \R^n$ is a cube and $c \in (0, \infty)$, we denote $cQ = x + (-cr, cr)^n$.

\begin{lemma}\label{lem:reverse_holder}
	Suppose that $f \in W^{1,n}_\loc (\Omega, \R^n)$ and $Df \in \mathcal M_n (K, \Sigma)$ with
	\begin{align*}
		K \in L^\infty (\Omega) \qquad \text{and} \qquad \Sigma \in L^1_\loc(\Omega).
	\end{align*}
	Then for every cube $Q$ such that $\overline{2Q} \subset \Omega$, we have the reverse H\"older inequality
	\[
		\left( \dashint_Q \abs{Df}^n \right)^\frac{n}{n+1}
		\leq C(n) \norm{K}_{L^\infty(\Omega)}^\frac{n}{n+1} \left(  \dashint_{2Q} \abs{Df}^\frac{n^2}{n+1} + \left( \dashint_{2Q} \Sigma \right)^\frac{n}{n+1} \right).
	\]
\end{lemma}
\begin{proof}
	Let $Q$ be such a cube. Choose a cutoff function $\eta \in C^\infty_0(2Q)$ such that $0 \leq \eta \leq 1$, $\eta \equiv 1$ on $Q$, and $\abs{\nabla \eta} \leq C_1(n) \abs{2Q}^{-1/n}$. By using the assumed $\abs{Df}^n \leq K J_f + \Sigma$ and $K \geq 1$ as well as a Caccioppoli-type inequality, we have
	\begin{align*}
		&\dashint_Q \abs{Df}^n 
		\leq \frac{\norm{K}_{L^\infty}}{\abs{Q}}
			\int_\Omega \frac{\abs{Df}^n \eta^n}{K} \\
		&\qquad\leq \frac{\norm{K}_{L^\infty(\Omega)}}{\abs{Q}} 
			\int_\Omega J_f \eta^n 
			+ \frac{\norm{K}_{L^\infty(\Omega)}}{\abs{Q}} 
			\int_\Omega \frac{\Sigma\eta^n}{K} \\
		&\qquad\leq \frac{C_2(n) \norm{K}_{L^\infty(\Omega)}}{\abs{Q}}
			\int_\Omega \abs{Df}^{n-1} \eta^{n-1} \abs{f-c} \abs{\nabla \eta}
			+ \frac{\norm{K}_{L^\infty(\Omega)}}{\abs{Q}} 
			\int_\Omega \Sigma\eta^n
	\end{align*}
	By using $\abs{\nabla \eta} \leq C_1(n) \abs{2Q}^{-1/n}$, $\eta \leq 1$, and $\abs{2Q} = 2^n\abs{Q}$, we hence obtain that
	\[
		\dashint_Q \abs{Df}^n 
		\leq C_3(n) \norm{K}_{L^\infty}  \left(
			\frac{1}{\abs{Q}^{\frac{1}{n}}}
				\dashint_{2Q} \abs{Df}^{n-1} \abs{f-c} 
			+ \dashint_{2Q} \Sigma \right).
	\]
	H\"older and Sobolev-Poincar\'e inequalities then yield that
	\begin{align*}
		&\frac{1}{\abs{Q}^{\frac{1}{n}}}
		\dashint_{2Q} \abs{Df}^{n-1} \abs{f-c}\\
		&\qquad \leq \left( \dashint_{2Q} \abs{Df}^\frac{n^2}{n+1}
			\right)^\frac{n^2 - 1}{n^2}
		\frac{1}{\abs{Q}^{\frac{1}{n}}} 
			\left( \dashint_{2Q} \abs{f-c}^{n^2} \right)^\frac{1}{n^2}\\
		&\qquad \leq \left( \dashint_{2Q} \abs{Df}^\frac{n^2}{n+1}
			\right)^\frac{n^2 - 1}{n^2}
		C_4(n) \left( \dashint_{2Q} \abs{Df}^\frac{n^2}{n+1} \right)^\frac{n+1}{n^2}\\
		&\qquad = C_4(n) \left(  \dashint_{2Q} 
			\abs{Df}^\frac{n^2}{n+1} \right)^\frac{n+1}{n}.
	\end{align*}
	Thus, the claimed estimate follows by using the elementary inequality $a + b \leq (a^{1/p}+b^{1/p})^p$ for $a,b\geq 0$, $p \geq 1$
\end{proof}

We then recall the statement of Theorem \ref{thm:bounded_K_higher_int_Lp} and give the short remaining parts of the proof.

\begin{customthm}{\ref{thm:bounded_K_higher_int_Lp}}
	For given $n\ge 2$ and $K_\circ \in [1, \infty)$, there exists a value $p(n, K_{\circ}) > 1$, such that if $f \in W_{\loc}^{1,n} (\Omega, \R^n)$ and $Df \in \mathcal M_n (K, \Sigma)$ with
	\begin{align*}
		\norm{K}_{L^\infty(\Omega)} \leq K_\circ \qquad \text{and} \qquad \Sigma \in L^{p}_\loc(\Omega),
	\end{align*}
	for some $p \in [1, p(n, K_{\circ}))$, then $\abs{Df}^n  \in  L^{p}_{\loc} (\Omega)$. 
\end{customthm}
\begin{proof}
	Due to $f$ satisfying the reverse H\"older inequality given in Lemma \ref{lem:reverse_holder}, the claimed result follows immediately from the version of Gehring's lemma given in \cite[Proposition 6.1]{Iwaniec-GehringLemma}. The upper bound of higher integrability given there depends only on the constants of the reverse H\"older inequality, which in turn depend only on $n$ and $\norm{K}_{L^\infty(\Omega)}$. 
\end{proof}

\subsection{Higher integrability on the Zygmund space scale}

For the Zygmund space version of our main result, we need a corresponding variant of Gehring's lemma. We expect this to be known, but are not aware of any references that would directly give the version we need. Hence, we provide a proof here of the relevant version of Gehring's lemma, with the proof modeled on the arguments used in \cite[Section 3]{Iwaniec-GehringLemma}. 

\begin{lemma}\label{lem:log_Gehring}
	Let $G, H \in L^p(\R^n)$ be non-negative functions satisfying the reverse H\"older inequality
	\[
		\left(\dashint_Q G^p\right)^\frac{1}{p}
		\leq C \left( \dashint_{2Q} G^q \right)^\frac{1}{q} + \left(\dashint_{2Q} H^p\right)^\frac{1}{p},
	\]
	for all cubes $Q \subset \R^n$, where $1 \leq q < p < \infty$ and $C \geq 1$ is a constant. Then for every $\mu > 0$, we have
	\[
		\int_{\R^n} G^p \log^\mu(e + G) \leq a \int_{\R^n} G^p + b\int_{\R^n} H^p \log^\mu(e + H),
	\]
	with $a = a(C, n, \mu, p, q) \geq 1$ and $b = b(C, n, \mu, p, q) \geq 1$.
\end{lemma}

We start the proof with the following estimate which directly follows from \cite[Section 3]{Iwaniec-GehringLemma}.

\begin{lemma}\label{lem:log_Gehring_level_set_estimate}
	Let $G, H \in L^p(\R^n)$ be non-negative functions satisfying the reverse H\"older inequality
	\[
		\left(\dashint_Q G^p\right)^\frac{1}{p}
		\leq C \left( \dashint_{2Q} G^q \right)^\frac{1}{q} + \left(\dashint_{2Q} H^p\right)^\frac{1}{p},
	\]
	for all cubes $Q \subset \R^n$, where $1 \leq q < p < \infty$ and $C \geq 1$ is a constant. Then for every $t > 0$, we have
	\begin{equation}\label{eq:level_set}
		\int_{G^{-1}(t, \infty)} G^p \leq \alpha t^{p - q} \int_{G^{-1}(t, \infty)} G^q + \beta \int_{H^{-1}(t, \infty)} H^p,
	\end{equation}
	with $\alpha = \alpha(n, C, p, q) > 1$ and $\beta = \beta(n, C, p, q) > 0$.
\end{lemma}
\begin{proof}
	This estimate is \cite[Proof of Lemma 3.1, estimate (3.11)]{Iwaniec-GehringLemma}, where in the notation used therein we've chosen $\Phi(t) = t F(t) = t^{p/q}$ with $F(t) = t^{p/q-1}$, $g = G^q$, and $h = H^q$.
\end{proof}

\begin{proof}[Proof of Lemma \ref{lem:log_Gehring}]
	We assume first that $\mu \neq 1$; for $\mu = 1$, see the remark at the end of the proof. We define an auxiliary function 
	\[
		A_\mu(t)=\frac{p- q}{\mu}\log^\mu(t) + \frac{\mu}{\mu-1}\log^{\mu - 1}(t)
	\]
	The purpose of this specific choice is that
	\begin{equation}\label{eq:A_mu_definition}
		t^{p-q} \log(t) A_\mu'(t) 
		= \frac{\dd}{\dd t} (t^{p-q} \log^\mu(t)).
	\end{equation}
	We may select a constant $M > 1$ large enough that $A_\mu$ and $A'_\mu$ are positive on $[M, \infty)$, and also large enough that
	\begin{equation}\label{eq:M_condition}
		A_\mu(t) - \frac{\alpha}{\log(M)} \log^\mu(t) \geq \frac{p-q}{2\mu} \log^\mu(e + t) \qquad \text{for all } t \in [M, \infty),
	\end{equation} 
	where $\alpha$ is from \eqref{eq:level_set}. Let $L > M$. We multiply both sides of \eqref{eq:level_set} with $A_\mu'(t)$, and integrate over $[M, L]$ with respect to $t$. By a use of the Fubini-Tonelli theorem, the left hand side yields
	\begin{multline*}
		\int_M^L A'_\mu(t) \int_{G^{-1}(t, \infty)} G^p(x) \dd x \dd t\\
		= \int_{G^{-1}(L, \infty)} G^p(x) \int_M^L A'_\mu(t) \dd t \dd x
		+ \int_{G^{-1}[M, L]} G^p(x) \int_M^{G(x)} A'_\mu(t) \dd t \dd x\\
		= \int_{G^{-1}(L, \infty)} A_\mu(L) G^p + \int_{G^{-1}[M, L]} G^p A_\mu(G) -  \int_{G^{-1}(M, \infty)} A_\mu(M) G^p.
	\end{multline*}
	By the same computation for the $H^p$-term, we get the upper bound 
	\begin{multline*}
	 	\int_M^L A'_\mu(t) \int_{H^{-1}(t, \infty)} H^p(x) \dd x \dd t\\
		= \int_{H^{-1}(L, \infty)} A_\mu(L) H^p + \int_{H^{-1}[M, L]} H^p A_\mu(H) -  \int_{H^{-1}(M, \infty)} A_\mu(M) H^p\\
	 	\leq  \int_{H^{-1}(L, \infty)} A_\mu(H) H^p + \int_{H^{-1}[M, L]} H^p A_\mu(H) \leq 2\int_{H^{-1}[M, \infty)} H^p A_\mu(H).
	\end{multline*}
	For the $G^q$-term, we use \eqref{eq:A_mu_definition} and similar computations to obtain that
	\begin{multline*}
		\int_M^L A'_\mu(t) t^{p-q} \int_{G^{-1}(t, \infty)} G^q(x) \dd x \dd t\\
		\leq \frac{1}{\log(M)} \int_M^L A'_\mu(t) t^{p-q} \log(t) \int_{G^{-1}(t, \infty)} G^q(x) \dd x \dd t\\
		= \frac{1}{\log(M)} \int_M^L \frac{\dd}{\dd t} (t^{p-q} \log^\mu(t)) \int_{G^{-1}(t, \infty)} G^q(x) \dd x \dd t\\
		\leq  \int_{G^{-1}(L, \infty)} \frac{L^{p-q} \log^{\mu}(L) G^q}{\log(M)} +  \int_{G^{-1}[M, L]} \frac{G^p \log^{\mu}(G)}{\log(M)}.
	\end{multline*}
	In total, we have
	\begin{multline*}
		 \int_{G^{-1}(L, \infty)} A_\mu(L) G^p + \int_{G^{-1}[M, L]} G^p A_\mu(G)\\
		\leq A_\mu(M) \int_{G^{-1}(M, \infty)} G^p + \frac{\alpha}{\log(M)} \int_{G^{-1}(L, \infty)}  L^{p-q} \log^{\mu}(L) G^q\\
		 + \frac{\alpha}{\log(M)} \int_{G^{-1}[M, L]} G^p \log^{\mu}(G) + 2\beta \int_{H^{-1}[M, \infty)} H^p A_\mu(H)
	\end{multline*}
	Note that on $G^{-1}(L, \infty)$, we have $L^{p-q} \leq G^{p-q}$. By applying this and subtracting the $\alpha/\log(M)$-terms from both sides of the above estimate, we obtain
	\begin{multline*}
		\int_{G^{-1}(L, \infty)} \left(A_\mu(L) - \frac{\alpha \log^\mu(L)}{\log(M)}\right) G^p + \int_{G^{-1}[M, L]} \left(A_\mu(G) - \frac{\alpha \log^\mu(G)}{\log(M)}\right) G^p\\
		\leq A_\mu(M) \int_{G^{-1}(M, \infty)} G^p + 2\beta \int_{H^{-1}[M, \infty)} H^p A_\mu(H)
	\end{multline*}
	We then apply \eqref{eq:M_condition}, and conclude that 
	\begin{multline*}
		\int_{G^{-1}[M, L]} G^p \log^\mu(e + G)\\
		\leq \int_{G^{-1}(L, \infty)} G^p \log^\mu(e + L) 
		+ \int_{G^{-1}[M, L]} G^p \log^\mu(e + G)  \\
		\leq \frac{2\mu A_\mu(M)}{p-q} \int_{G^{-1}(M, \infty)} G^p + \frac{4\mu \beta}{p-q} \int_{H^{-1}[M, \infty)} H^p A_\mu(H)
	\end{multline*}
	Notably, this upper bound is independent on $L$. Since we have $0 \leq A_\mu(H) \leq A_\mu(e + H) \leq ((p-q)/\mu + \mu/\abs{\mu - 1}) \log^\mu(e + H)$ in $H^{-1}[M, \infty)$, letting $L \to \infty$ gives us
	\[
		\int_{G^{-1}[M, \infty)} G^p \log^\mu(e + G)
		\leq a_0 \int_{\R^n} G^p + b \int_{\R^n} H^p \log^\mu(e + H),	
	\]
	with $a_0, b$ dependent only on $\alpha, \beta, p, q, \mu$. The final desired claim then follows by combining the previous estimate with
	\[
		\int_{G^{-1}[0, M)} G^p \log^\mu(e + G) \leq \log^\mu(e+M) \int_{\R^n} G^p.
	\]
	
	We finally comment on the case $\mu = 1$. In this case, we must instead define $A_1(t) = (p-q)\log(t) + \log \log(t)$, which yields \eqref{eq:A_mu_definition} for $\mu = 1$. The rest of the proof goes through essentially similarly in this case. 
\end{proof}

With Lemma \ref{lem:log_Gehring} proven, we may proceed to prove Theorem \ref{thm:bounded_K_higher_int}. We again recall the statement.

\begin{customthm}{\ref{thm:bounded_K_higher_int}}
	Suppose that $f \in W_{\loc}^{1,n} (\Omega, \R^n)$ and $Df \in \mathcal M_n (K, \Sigma)$ with
	\begin{align*}
		K \in L_{\loc}^\infty (\Omega) \qquad \text{and} \qquad \Sigma \log^{\mu} \left( e + \Sigma \right) \in L^1_\loc(\Omega),
	\end{align*}
	for some $\mu \ge 0$. Then $\abs{Df}^n \log^\mu (e+\abs{Df}) \in  L^1_{\loc} (\Omega)$.
\end{customthm}
\begin{proof}
	We select a ball $B = \B^n(x_0, r)$ with $r \leq 1$ such that $\overline{B} \subset \Omega$. By using Lemma \ref{lem:reverse_holder}, we obtain that $\abs{Df}$ and $\Sigma$ satisfy a reverse H\"older inequality for all cubes $Q$ with $2Q \subset B$. It was shown in the proof of \cite[Proposition 6.1]{Iwaniec-GehringLemma} that in this case, the functions $G, H \colon \R^n \to [0, \infty)$ defined by
	\begin{align*}
		G(x) &= \dist(x, \R^n \setminus B) \abs{Df(x)}^\frac{n^2}{n+1}\\ 
		H(x) &= \dist(x, \R^n \setminus B) \Sigma^\frac{n}{n+1} + \chi_{B}(x) \left( \int_{B} \Sigma \right)^\frac{n}{n+1}
	\end{align*}
	satisfy a reverse H\"older inequality in all of $\R^n$. In particular, it follows from Lemma \ref{lem:log_Gehring} that
	\begin{equation}\label{eq:alt_functions_reverse_holder}
		\int_{\R^n} G^\frac{n+1}{n} \log^\mu(e + G) \leq a \int_{\R^n} G^\frac{n+1}{n} + b \int_{\R^n} H^{\frac{n+1}{n}} \log^\mu(e + H).
	\end{equation}
	
	We then assume $0 < \eps < r$, and denote $B_\eps = \{x \in B : \dist(x, R^n \setminus B) > \eps\}$. We note that for $t \geq 1$, and $p \geq 1$, we may estimate using Bernoulli's inequality that
	\[
		e + \eps t^p = e(1 + \eps e^{-1} t^p) \geq e^\eps (1 + e^{-1} t^p)^\eps = (e + t^p)^\eps \geq (e + t)^\eps.
	\] 
	Hence, for every point $x \in B_\eps$, we have either $\abs{Df(x)} \leq 1$ and thus also $\abs{Df(x)}^n \log^\mu\left(e + \abs{Df(x)}\right) \leq \log^\mu(e + 1)$, or
	\begin{align*}
		G^\frac{n+1}{n} \log^\mu(e + G) 
		&\geq \eps^{\frac{n+1}{n}} \abs{Df}^n \log^\mu(e + \eps \abs{Df}^\frac{n^2}{n+1})\\
		&\geq \eps^{\frac{n+1}{n} + \mu} \abs{Df}^n \log^\mu(e + \abs{Df}).
	\end{align*}
	Consequently,
	\[
		\int_{B_\eps} \abs{Df}^n \log^\mu(e + \abs{Df}) \leq \abs{B_\eps} \log^\mu(e + 1) + \eps^{-\frac{n+1}{n} - \mu} \int_{\R^n} G^\frac{n+1}{n} \log^\mu(e + G).
	\]
	
	On the other hand, $G^{(n-1)/n} \leq \abs{Df}^n \chi_{B} \in L^1(\R^n)$. For the $H$-term of \eqref{eq:alt_functions_reverse_holder}, we have $H \equiv 0$ outside $B$ and $H \leq \Sigma^{n/(n+1)} + C$ in $B$ with $C = \norm{\Sigma}_{L^1(B)}^{n/(n+1)} < \infty$. We recall that we have the elementary inequality $\log(e + a + b) \leq \log(e + a) + \log(e + b)$ for $a, b \geq 0$, and that $\Sigma^{n/(n+1)} \leq 1 + \Sigma$. Hence, we may estimate
	\begin{multline*}
		\int_{\R^n} H^{\frac{n+1}{n}} \log^\mu(e + H)
		\leq \int_B (\Sigma^\frac{n}{n+1} + C)^{\frac{n+1}{n}}
			\log^\mu \bigl( e + \Sigma^\frac{n}{n+1} + C \bigr)\\
		\leq \int_B 2^{\frac{n+1}{n} + \mu} \bigl(\Sigma + C^{\frac{n+1}{n}} \bigr) \bigl(\log^\mu(e + \Sigma) + \log^\mu(e + C + 1) \bigr) < \infty.
	\end{multline*}
	It follows that $\abs{Df}^n \log^\mu(e + \abs{Df})$ has finite integral over $B_\eps$, which completes the proof of the claim.
\end{proof}

\subsection{Embedding theorems}

As stated in the introduction, Theorem \ref{thm:bounded_K} is a direct corollary of combining Theorem \ref{thm:bounded_K_higher_int} with a suitable version of Morrey's inequality for Zygmund spaces. Recall that the classical Morrey's inequality implies that if $p > n$, then elements of $W^{1,p}_\loc(\Omega)$ have a locally H\"older continuous representative, with H\"older exponent $1 - n/p$. For a Zygmund space version, we refer to e.g.\ \cite[Theorem 3.1]{Iwaniec-Onninen_ContEstForN-Harm}, which gives us the following.

\begin{thm}\label{thm:Zygmund_Morrey}
	Let $\Omega \subset \R^n$ be a domain, and let $f \in W^{1,n}_\loc(\Omega)$ satisfy
	\[
		\abs{Df}^n \log^\mu(e + \abs{Df}) \in L^1_\loc(\Omega),
	\]
	where $\mu > n-1$. Then $f$ has a continuous representative. In particular, whenever $0 < r < R$ and $\overline{\B^n(x, R)} \subset \Omega$, the modulus of continuity $\omega_f (x_0,r)$ defined in \eqref{eq:mod_cont} satisfies
	\[
		\omega_f(x, r) \leq C(Df, \mu, x, R) \log^\frac{\mu - n + 1}{n}\left( 1 + \frac{2R}{r} \right),
	\]
	where
	\[
		C(Df, \mu, x, R) = \dashint_{\B^n(x, R)} \abs{Df}^n \log^\mu \left( e + \frac{\abs{\B^n(x, r)} \abs{Df}}{\norm{Df}^n_{L^n(\B^n(x, R))}} \right).
	\] 
\end{thm}

Hence, by combining Theorems \ref{thm:bounded_K_higher_int} and \ref{thm:Zygmund_Morrey}, the proof of Theorem \ref{thm:bounded_K} is complete.

Due to us requiring it in the following subsection, we also recall the corresponding result for $\mu \in [0, n-1)$. In this case, $f$ is not necessarily continuous, but does satisfy an exponential Sobolev embedding theorem. We refer to e.g.\ \cite[Theorem 2, Example 1]{Cianchi_Orlicz-Trudinger} for the following result; note also that the case $\mu = 0$ corresponds to the classical Moser-Trudinger inequality.

\begin{thm}\label{thm:Zygmund_Sobolev_embedding}
	Let $\Omega \subset \R^n$ be a domain, and let $f \in W^{1,n}_\loc(\Omega)$ satisfy
	\[
		\abs{Df}^n \log^\mu(e + \abs{Df}) \in L^1_\loc(\Omega),
	\]
	where $0 \leq \mu < n-1$. Then there exists $\lambda > 0$ such that
	\[
		\exp(\lambda \abs{f}^\frac{n}{n-1-\mu}) \in L^1_\loc(\Omega).
	\]
\end{thm} 

\subsection{Continuity for \eqref{eq:MFDvalue} with bounded $K$}

The final result we prove in this section is Theorem \ref{thm:K_bounded_y_0}. For the proof, we require the following lemma on the integrability of products of functions.

\begin{lemma}\label{lem:product_int_lemma}
	Let $\Omega \subset \R^n$ be measurable, let $\mu, \nu, \lambda > 0$ be such that $\nu \leq \mu$, and let $f, g \colon \Omega \to [0, \infty]$ be measurable functions such that
	\[
	f \log^\mu (e + f) \in L^1_\loc(\Omega), \qquad \exp\bigl(\lambda g^\frac{1}{\nu}\bigr) \in L^1_\loc(\Omega). 
	\]
	Then $fg \log^{\mu - \nu} (e + fg) \in L^1_\loc(\Omega)$.
\end{lemma}

We begin the proof of Lemma \ref{lem:product_int_lemma} by recalling the proof of the following elementary inequality. See e.g.\ \cite[Lemmas 2.7, 6.2]{Hencl-Koskela-book} for similar results and proofs.

\begin{lemma}\label{lem:exp_Young}
	Let $a, b \geq 0$, and $\kappa, \lambda > 0$. Then
	\[
	ab < \exp\left(\lambda a^{\frac{1}{\kappa}}\right)
	+  C(\kappa, \lambda) b \log^{\kappa}(e + b),
	\]
	where $C(\kappa, \lambda) \geq 0$.
\end{lemma}
\begin{proof}
	Note that there exists a constant $A = A(\kappa)$ such that 
	\begin{equation}\label{eq:exp_taylor_est}
		\exp(t) \geq A t^{2\kappa}.
	\end{equation} 
	If $ab \leq \exp(a^{1/\kappa} \lambda)$, then the claim is clear. Hence, we assume that $ab > \exp(a^{1/\kappa} \lambda)$, with a goal of showing that $ab < C(\kappa, \lambda) b \log^{\kappa}(e + b)$. 
	
	By combining this assumption with \eqref{eq:exp_taylor_est}, we have
	\[
	a = a^{-1} a^2 \leq a^{-1} \left( \frac{\exp(a^{1/\kappa} \lambda)}{A\lambda^{2\kappa}}  \right) < a^{-1} \left( \frac{ab}{A\lambda^{2\kappa}} \right) = \frac{b}{A\lambda^{2\kappa}} .
	\]
	Consequently, we have
	\[
	\exp(a^{1/\kappa} \lambda) < ab < \frac{b^2}{A\lambda^{2\kappa}} < \frac{(e + b)^2}{A\lambda^{2\kappa}}.
	\]
	Taking logarithms yields
	\[
	a^{1/\kappa} \lambda < \log \frac{(e + b)^2}{A\lambda^{2\kappa}} = \log \frac{1}{A\lambda^{2\kappa}} + 2 \log (e + b).
	\]
	In particular
	\[
	a < \abs{ \frac{1}{\lambda} \log \frac{1}{A\lambda^{2\kappa}} + \frac{2}{\lambda} \log (e + b) }^{\kappa}
	\leq \frac{2^\kappa}{\lambda^\kappa} \abs{\log \frac{1}{A\lambda^{2\kappa}}}^\kappa + \frac{4^\kappa}{\lambda^\kappa} \log^\kappa (e + b).
	\]
	And hence, we obtain the desired estimate
	\begin{align*}
		ab &< \left( \frac{2^\kappa}{\lambda^\kappa} \abs{\log \frac{1}{A\lambda^{2\kappa}}}^\kappa \right) b + \left(\frac{4^\kappa}{\lambda^\kappa}\right) b \log^\kappa(e+b)\\
		&\leq \left( \frac{2^\kappa}{\lambda^\kappa} \abs{\log \frac{1}{A\lambda^{2\kappa}}}^\kappa + \frac{4^\kappa}{\lambda^\kappa} \right) b \log^\kappa (e+b).
	\end{align*}
\end{proof}
\begin{proof}[Proof of Lemma \ref{lem:product_int_lemma}]
	We first observe that $fg \in L^1_\loc(\Omega)$. Indeed, Lemma \ref{lem:exp_Young} yields that $fg \leq \exp(\lambda g^{1/\nu}) +  C f \log^{\nu}(e + f)$, where both terms on the right hand side are integrable by $\nu \leq \mu$.
	
	We then further estimate using Lemma \ref{lem:exp_Young}
	\begin{multline}\label{eq:sigma_tilde_est}
		fg \log^{\mu - \nu} (e + fg) \\
		\leq \exp\left(2^{-1} \lambda g^\frac{1}{\nu} \right) \log^{\mu - \nu} (e + fg) + C_1 f \log^{\nu} (e + f) \log^{\mu - \nu} (e + fg).
	\end{multline}
	We have $\exp\left(2^{-1} \lambda g^\frac{1}{\nu} \right) \in L^2_\loc(\Omega)$, and also $\log^{\mu} (e + f) \in L^2_\loc(\Omega)$. It follows that the first term on the right hand side of \eqref{eq:sigma_tilde_est} is locally integrable. For the second term, we estimate $e + fg \leq (e + f) (e + g)$, and hence
	\begin{multline}\label{eq:sigma_tilde_est_2}
		f \log^{\nu} (e + f) \log^{\mu - \nu} (e + fg)\\
		\leq C_2 \left( f \log^{\mu} (e + f) + f \log^{\nu} (e + f) \log^{\mu - \nu} (e + g) \right).
	\end{multline}
	The first term on the right hand side of \eqref{eq:sigma_tilde_est_2} is locally integrable by assumption. For the second term, we again use Lemma \ref{lem:exp_Young}, this time with $\kappa = \mu-\nu$ and $\lambda = 1$. We get
	\begin{align*}
		&f \log^{\nu} (e + f) \log^{\mu - \nu} (e + g)\\
		&\quad \leq  e + g + C_3 f \log^{\nu} (e + f) \log^{\mu-\nu}(e + f \log^{\mu-\nu}(e + f))\\
		&\quad \leq e + g + C_4 f \log^{\mu} (e + f)\\
		&\qquad\qquad + C_4 f \log^{\nu} (e + f) \log^{\mu-\nu}(e + \log^{\mu-\nu}(e + f)), 
	\end{align*}
	where the right hand side is locally integrable by the local integrability of $g$ and $f \log^{\mu} (e + f)$. Hence, the claim follows.
\end{proof}

We're now ready to prove Theorem \ref{thm:K_bounded_y_0}. We again recall the statement for convenience.

\begin{customthm}{\ref{thm:K_bounded_y_0}}
	Let $\Omega$ be a domain in $\R^n$. Suppose that a Sobolev mapping $f \in W_{\loc}^{1,n} (\Omega, \R^n)$ satisfies $Df \in \cM_n(K, \Sigma, y_0)$ with $K \colon \Omega \to [1, \infty)$, $\Sigma \colon \Omega \to [0, \infty)$ and $y_0 \in \R^n$.  If 
	\begin{align*}
		K \in L_{\loc}^\infty (\Omega) \qquad \text{and} \qquad \Sigma \log^{\mu} \left( e + \Sigma \right) \in L^1_\loc(\Omega),
	\end{align*}
	for some $\mu > n-1$, then $f$ has a continuous representative. 
\end{customthm}

\begin{proof}
	Let $q = \mu - (n-1) > 0$. By slightly shrinking $\mu$, we may assume that $n-1$ is not an integer multiple of $q$. By our assumption, we have $\abs{Df}^n \leq K J_f + \Sigma'$, where $\Sigma' = \Sigma \abs{f - y_0}^n$.
	
	By the Moser-Trudinger inequality (case $q = 0$ of Theorem \ref{thm:Zygmund_Sobolev_embedding}), there exists $\lambda_0 > 0$ such that
	\[
		\exp\left(\lambda_0 \abs{f-y_0}^{\frac{n}{n-1}}\right) \in L^1_\loc(\Omega).
	\]
	Combining this with our assumption that $\Sigma \log^{\mu} (e + \Sigma) \in L^1_\loc(\Omega)$ and recalling that $q = \mu - (n-1)$, we can thus use Lemma \ref{lem:product_int_lemma} to conclude that
	\[
		\Sigma' \log^q (e + \Sigma') \in L^1_\loc(\Omega).
	\]
	Using Theorem \ref{thm:bounded_K_higher_int}, we hence conclude that
	\[
		\abs{Df}^n \log^q (e + \abs{Df}) \in L^1_\loc(\Omega).
	\]
	
	If $q > n-1$, we are now done, since Theorem \ref{thm:Zygmund_Morrey} implies that $f$ has a continuous representative. Otherwise, we proceed to iterate this argument. Indeed, since $\abs{Df}^n \log^q (e + \abs{Df}) \in L^1_\loc(\Omega)$, Theorem \ref{thm:Zygmund_Sobolev_embedding} yields us a slightly better estimate 
	\[
		\exp\left(\lambda_1 \abs{f-y_0}^{\frac{n}{(n-1) - q}}\right) 
		\in L^1_\loc(\Omega)
	\]
	for some $\lambda_1 >0$. Lemma \ref{lem:product_int_lemma} then yields that
	\[
		\Sigma' \log^{2q} (e + \Sigma') 
		= \Sigma' \log^{\mu - ((n-1) - q)} (e + \Sigma') 
		\in L^1_\loc(\Omega),
	\]
	from which we get that
	\[
		\abs{Df}^n \log^{2q} (e + \abs{Df}) \in L^1_\loc(\Omega).
	\]
	Then next iteration of this argument then yields $\abs{Df}^n \log^{3q} (e + \abs{Df}) \in L^1_\loc(\Omega)$, the next iteration after that yields $\abs{Df}^n \log^{4q} (e + \abs{Df}) \in L^1_\loc(\Omega)$, et cetera.
	
	We may continue this iteration until $\abs{Df}^n \log^{kq} (e + \abs{Df}) \in L^1_\loc(\Omega)$, where $k$ is the smallest positive integer such that $kq > n-1$. Indeed, we assumed $n-1$ not to be an integer multiple of $q$, so $(k-1)q$ is a valid exponent for Theorem \ref{thm:Zygmund_Sobolev_embedding}. Moreover, we also must have $kq < \mu$, since $kq = (k-1)q + q < (n-1) + q = \mu$. Hence, it follows that $f$ has a continuous representative by Theorem \ref{thm:Zygmund_Morrey}.
\end{proof}

\section{Direct continuity results}

In this section, we prove Theorems \ref{thm:exp_cont_modulus} and \ref{thm:MFDvalue_cont}. The method is a generalization of the approach used in \cite[Section 3]{Kangasniemi-Onninen_Heterogeneous}. In particular, we prove a decay estimate for the integral of $\abs{Df}^n$ over balls, which then implies continuity by using a chain of balls argument as in \cite{Hajlasz-Koskela-SobolevmetPoincare}.

We begin by recalling an estimate that is used in the proofs of similar continuity results for mappings of finite distortion; see e.g.\ \cite[Section 3]{Onninen-Zhong_MFD-mod-of-continuity} or \cite[Theorem 5.18]{Hencl-Koskela-book}. We give the proof for the convenience of the reader.

\begin{lemma}\label{lem:triple_Jensen_K_estimate}
	Let $\Omega \subset \R^n$ be a domain, and let $\exp(K) \in L^\lambda(\Omega)$ with $\lambda > 0$. Then there exist constants $C = C(\Omega, K, \lambda) > 0$ and $R_0 = R_0(\Omega, K, \lambda)$ as follows: if $x \in \Omega$, $R < \min(R_0, d(x, \partial \Omega))$ and $r \in (0, R/e^3)$, then
	\[
	\int_r^R s^{-1} \left( \dashint_{\partial \B^n(x, s)} K^{n-1} \right)^{-\frac{1}{n-1}} \dd s
	\geq \frac{\lambda}{n} \left( \log \log \frac{C}{r^n} - \log \log \frac{Ce^2}{R^n} \right).
	\]
\end{lemma}
\begin{proof}
	We denote $B_s = \B^n(x, s)$. We let $k$ be the largest integer such that $r e^k \leq R$. Since $r < R/e^3$, we must have $k \geq 3$. We define the function $\tilde{K} = \max(K, (n-2)\lambda^{-1})$, where we still have $\exp(\tilde{K}) \in L^\lambda_\loc(\Omega)$. We estimate $K \leq \tilde{K}$, perform a change of variables, and split the integral into a sum as follows:
	\begin{multline*}
		\int_r^R s^{-1} \left( \dashint_{\partial B_s} K^{n-1} \right)^{-\frac{1}{n-1}} \dd s
		\geq \sum_{i = 1}^{k-1} \int_{i + \log r}^{i + 1 + \log r} \left( \dashint_{\partial B_{e^t}} \tilde{K}^{n-1} \right)^{-\frac{1}{n-1}} \dd t.
	\end{multline*}
	
	We then use Jensen's inequality a total of three times, with the convex functions $\tau \mapsto \tau^{-1}$, $\tau \mapsto \exp(\lambda\tau^{\frac{1}{n-1}})$ and $\tau \mapsto \exp(\tau)$. Note that $\tau \mapsto \exp(\lambda\tau^{\frac{1}{n-1}})$ is only convex for $\tau \geq ((n-2)\lambda^{-1})^{n-1}$, but the range of $\tilde{K}^{n-1}$ is in this region. The resulting estimate is
	\begin{multline*}
		\int_{i + \log r}^{i + 1 + \log r} \left( \dashint_{\partial B_{e^t}} \tilde{K}^{n-1} \right)^{-\frac{1}{n-1}} \dd t
		\geq 
		\left(\int_{i + \log r}^{i + 1 + \log r} \left( \dashint_{\partial B_{e^t}} \tilde{K}^{n-1} \right)^{\frac{1}{n-1}} \dd t \right)^{-1}\\
		\geq 
		\lambda \left(\int_{i + \log r}^{i + 1 + \log r} \log \left( \dashint_{\partial B_{e^t}} \exp(\lambda \tilde{K}) \right) \dd t \right)^{-1}\\
		\geq \lambda \log^{-1} \int_{i + \log r}^{i + 1 + \log r}  \left(\dashint_{\partial B_{e^t}} \exp(\lambda \tilde{K})\right) \dd t	
	\end{multline*}
	Now we may estimate
	\begin{multline*}
		\log^{-1} \int_{i + \log r}^{i + 1 + \log r}  \left(\dashint_{\partial B_{e^t}} \exp(\lambda \tilde{K})\right) \dd t
		= \log^{-1} \int_{re^{i}}^{re^{i+1}} 
		\left( \dashint_{\partial B_{s}} \exp(\lambda \tilde{K}) \right) \frac{\dd s}{s}\\
		= \log^{-1} \int_{re^{i}}^{re^{i+1}} \frac{1}{\omega_{n-1} s^{n}} \left( \int_{\partial B_{s}} \exp(\lambda \tilde{K}) \right) \dd s
		\geq \log^{-1} \frac{\smallnorm{\exp(\lambda \tilde{K})}_{L^1(\Omega)}}{\omega_{n-1} (re^{i})^{n}}.
	\end{multline*}
	We then select $C = \smallnorm{\exp(\lambda \tilde{K})}_{L^1(\Omega)}/ \omega_{n-1}$. The sum of the above terms over $i$ can now be estimated by
	\begin{multline*}
		\sum_{i = 1}^{k-1} \int_{i + \log r}^{i + 1 + \log r} \left( \dashint_{\partial B_{e^t}} \tilde{K}^{n-1} \right)^{-\frac{1}{n-1}} \dd t
		\geq 
		\lambda \sum_{i = 1}^{k-1} \log^{-1} \frac{C}{(re^{i})^{n}}\\
		\geq
		\lambda \int_{0}^{k-1} \log^{-1} \frac{C}{(re^{t})^{n}} \dd t 
		\geq 
		\lambda \int_{r}^{R/e^2} s^{-1} \log^{-1} \frac{C}{s^{n}} \dd s\\
		= \frac{\lambda}{n}\left( \log \log \frac{C}{r^{n}} - \log \log \frac{Ce^2}{R^{n}} \right).
	\end{multline*}
	The claim hence holds, assuming that $\log \log (Ce^2/R^{n})$ is well defined; this is the case if we select $R_0^n = Ce$.
\end{proof}

We then consider the following abstract differential inequality of real functions, and show that it yields a decay condition. This is a more general version of \cite[Lemma 3.2]{Kangasniemi-Onninen_Heterogeneous}, which is essentially given by the case $\Psi(r) = r$ and $\Gamma(r) = C r^\alpha$.

\begin{lemma}\label{lem:generalized_diff_ineq}
	Let $A > 0$, and let $\Phi \colon [0, R] \to [0, S]$, $\Psi \colon [0, R] \to [0, \infty)$, and $\Gamma \colon [0, R] \to [0, \infty)$ be absolutely continuous increasing functions such that $\Phi(0) = 0$. Suppose that
	\[
	\Phi(r) \leq A \frac{\Psi(r)}{\Psi'(r)} \Phi'(r) + \Gamma(r)
	\]
	for a.e.\ $r \in (0, R)$, where $A > 0$. Then there exists a constant $C = C(A, R, S, \Psi, \Gamma) \geq 0$ such that we have
	\[
	\Phi(r) \leq \Gamma(r) + C \Psi^{A^{-1}}(r) \left( 1 + \int_r^R \frac{\Gamma'(s)}{\Psi^{A^{-1}}(s)} \dd s\right)
	\]
	for all $r \in [0, R]$.
\end{lemma}
\begin{proof}
	We find an integrating factor for the terms involving $\Phi$:
	\begin{align*}
		-\frac{d}{ds} (\Psi^{-A^{-1}}(s) \Phi(s))
		&= \left( \Phi(s) - A \frac{\Psi(s)}{\Psi'(s)} \Phi'(s) \right) \left( A^{-1} \Psi^{-A^{-1} - 1}(s) \Psi'(s)\right)\\
		&\leq - \Gamma(s) \frac{d}{ds} \Psi^{-A^{-1}}(s).
	\end{align*}
	We then integrate on both sides over $[r, R]$, and use integration by parts:
	\begin{multline*}
		\Psi^{-A^{-1}}(r) \Phi(r) - \Psi^{-A^{-1}}(R) \Phi(R)
		\leq - \int_r^R \Gamma(s) \left(\frac{d}{ds} \Psi^{-A^{-1}}(s)\right) \dd s\\
		= \Gamma(r) \Psi^{-A^{-1}}(r) - \Gamma(R) \Psi^{-A^{-1}}(R) + \int_r^R \frac{\Gamma'(s)}{\Psi^{A^{-1}}(s)} \dd s.
	\end{multline*}
	Multiplying by $\Psi^{A^{-1}}(r)$ and moving the negative term to the right hand side yields the desired
	\[
	\Phi(r) \leq \Gamma(r) + \frac{S - \Gamma(R)}{\Psi^{A^{-1}}(R)} \Psi^{A^{-1}}(r) +  \Psi^{A^{-1}}(r) \int_r^R \frac{\Gamma'(s)}{\Psi^{A^{-1}}(s)}  \dd s.
	\]
\end{proof}

Combining Lemmas \ref{lem:triple_Jensen_K_estimate} and \ref{lem:generalized_diff_ineq} allows us to show the following decay estimate.

\begin{lemma}\label{lem:DfperK_decay}
	Let $\Omega \subset \R^n$ be a connected domain. Let $f \in W^{1,n}(\Omega, \R^n)$ and $Df \in \mathcal M_n (K, \Sigma)$ with $\exp(\lambda K) \in L^1(\Omega)$ and $\Sigma \log^{\mu} \left( e + \Sigma \right) \in L^1(\Omega)$, where $\mu > \lambda > 0$. Then there exists $R_0 = R_0(\Omega, \lambda, K) > 0$ as follows: for any choice of $x \in \Omega$ and $R \in (0, R_0)$ such that $\B^n(x, R) \subset \Omega$, we have 
	\[
	\int_{\B^n(x, r)} \frac{\abs{Df}^n}{K}
	\leq C \log^{-\lambda} r^{-1}.
	\] 
	for all $r \in (0, R/e^3)$, where $C = C(\Omega, \lambda, \mu, K, \Sigma, R, \norm{Df}_{L^n(\Omega)})$. Notably, $C$ is independent of $x$.
\end{lemma}
\begin{proof}
	We choose $R_0 < e^{-1}$ such that Lemma \ref{lem:triple_Jensen_K_estimate} holds for $R < R_0$: note that this choice depends only on $\Omega$, $\lambda$, and $K$. We fix a point $x \in \Omega$ and a radius $R < R_0$ such that $\B^n(x, R) \subset \Omega$, and we denote $B_r = \B^n(x, r)$ for all $r \in [0, R]$. We then define a function $\Phi \colon [0, R] \to [0, \infty)$ by
	\[
	\Phi(r) = \int_{B_r} \frac{\abs{Df}^n}{K}.
	\]
	By using the definition of $\mathcal{M}_n(K, \Sigma)$, we may estimate $\Phi(r)$ by
	\[
	\int_{B_r} \frac{\abs{Df}^n}{K} \leq \int_{B_r} J_f + \int_{B_r} \frac{\Sigma}{K}.
	\]
	
	We apply the isoperimetric inequality for Sobolev maps on the first term, followed by a use of H\"older's inequality. The result is
	\begin{align*}
		\int_{B_r} J_f 
		&\leq \frac{1}{n\sqrt[n-1]{\omega_{n-1}}} \left( \int_{\partial B_r} \abs{Df}^{n-1} \right)^\frac{n}{n-1}\\
		&\leq \frac{1}{n\sqrt[n-1]{\omega_{n-1}}} \left( \int_{\partial B_r} K^{n-1} \right)^\frac{1}{n-1} \int_{\partial B_r} \frac{\abs{Df}^n}{K}\\
		&= \frac{r}{n} \left( \dashint_{\partial B_r} K^{n-1} \right)^\frac{1}{n-1} \int_{\partial B_r} \frac{\abs{Df}^n}{K}.
	\end{align*}
	For the other term, using $K \geq 1$, $\Sigma \log^{\mu}_{+} \left( \Sigma \right) \in L^1(\Omega)$, and $r \leq R < R_0 < e^{-1}$, we estimate that
	\begin{multline*}
		\int_{B_r} \frac{\Sigma}{K} \leq \int_{B_r} \Sigma
		\leq \int_{\{z \in B_r : \Sigma(z) \leq r^{-1}\}} \Sigma
		+ \int_{\{z \in B_r : \Sigma(z) > r^{-1}\}} \Sigma\\
		\leq \int_{\{z \in B_r : \Sigma(z) \leq r^{-1}\}} r^{-1}
		+ \int_{\{z \in B_r : \Sigma(z) > r^{-1}\}} \frac{\Sigma \log^\mu \Sigma}{\log^\mu r^{-1}}\\
		\leq \frac{\vol_n(B_r)}{r} + \log^\mu r^{-1} \int_{B_r} \Sigma \log^{-\mu} (e + \Sigma)
		\leq C_1 \log^{-\mu} r^{-1}
	\end{multline*}
	for some $C_1 = C_1(n, \mu, \Sigma, R) \geq 0$. In conclusion, we have
	\begin{equation}\label{eq:before_Psi}
		\Phi(r) \leq \frac{r}{n} \left( \dashint_{\partial B_r} K^{n-1} \right)^\frac{1}{n-1} \Phi'(r) + C_1\log^{-\mu} r^{-1}
	\end{equation}
	for all $r \in (0, R)$.
	
	We then define
	\[
	\Psi(r) = \exp \left( - \int_r^R s^{-1} \left( \dashint_{\partial B_s} K^{n-1} \right)^{-\frac{1}{n-1}} \dd s \right).
	\]
	A simple computation by chain rule hence reveals that
	\[
	\Psi'(r) = \Psi(r) r^{-1} \left( \dashint_{\partial B_r} K^{n-1} \right)^{-\frac{1}{n-1}}
	\]
	for all $r \in (0, R)$. In particular, \eqref{eq:before_Psi} now reads as
	\begin{equation*}
		\Phi(r) \leq \frac{1}{n} \frac{\Psi(r)}{\Psi'(r)}\Phi'(r) + C_1 \log^{-\mu} r^{-1}.
	\end{equation*}
	We also note that since $K \geq 1$, we have $\Phi(r) \leq S$ for all $r \in [0, R]$ with $S = \norm{Df}_{L^n(\Omega)}$. Hence, we are in position to apply Lemma \ref{lem:generalized_diff_ineq}, which yields that
	\begin{equation}\label{eq:Phi_estimate}
		\Phi(r) \leq C_1\log^{-\mu} r^{-1} + C_2 \left( \Psi^n(r) + \int_r^R \frac{\Psi^n(r)}{\Psi^{n}(s)} \frac{\dd s}{s \log^{\mu + 1}(s^{-1})} \right)
	\end{equation}
	when $r \in [0, R/e^3]$, for some $C_2 = C_2(\Omega, \lambda, \mu, K, \Sigma, R, \norm{Df}_{L^n(\Omega)})$.
	
	Lemma \ref{lem:triple_Jensen_K_estimate} yields that for $r \in (0, R/e^3)$, we have 
	\begin{multline*}
		\Psi^n(r) 
		\leq \exp \left( - \frac{n\lambda}{n} \left( \log \log \frac{C_3}{r^{n}} - \log \log \frac{C_3e^2}{R^{n}}\right) \right)\\
		= \left( \frac{\log (C_3 e^2 R^{-n})}{\log (C_3 r^{-n})} \right)^{\lambda} \leq C_4 \log^{-\lambda} r^{-1},
	\end{multline*}
	where $C_3 = C_3(\Omega, \lambda, K)$ and $C_4 = C_4(\Omega, \lambda, K, R)$. Since $\mu > \lambda$, we also have
	\[
	\log^{-\mu} r^{-1} \leq C_5 \log^{-\lambda} r^{-1}
	\]
	for all $r \in (0, R/e^3]$, where $C_5 = C_5(\mu, \lambda, R)$. Hence, in order to obtain the claimed decay estimate for $\Phi$ from \eqref{eq:Phi_estimate}, it remains to estimate the term with the integral. 
	
	For this, we let $r \in (0, R/e^3)$, and split the integral into two parts:
	\begin{multline*}
		\int_r^R \frac{\Psi^n(r)}{\Psi^{n}(s)} \frac{\dd s}{s \log^{\mu + 1}(s^{-1})}\\
		=  \int_r^{e^3 r} \frac{\Psi^n(r)}{\Psi^{n}(s)} \frac{\dd s}{s \log^{\mu + 1}(s^{-1})}
		+ \int_{e^3 r}^R \frac{\Psi^n(r)}{\Psi^{n}(s)} \frac{\dd s}{s \log^{\mu + 1}(s^{-1})}.
	\end{multline*}
	In the range of the latter integral, we have $r < s/e^3$, which allows us to use Lemma \ref{lem:triple_Jensen_K_estimate} again to estimate
	\begin{align*}
		\frac{\Psi^n(r)}{\Psi^n(s)} &= \exp \left(- n\int_r^s t^{-1} \left( \dashint_{\partial B_t} K^{n-1} \right)^{-\frac{1}{n-1}} \dd t \right)\\
		&\leq \left( \frac{\log (C_3 e^2 s^{-n})}{\log (C_3 r^{-n})} \right)^{\lambda}.
	\end{align*} 
	Hence, by using the fact that $\mu - \lambda > 0$, and the fact that $s^{-1} \log^{-1-t}(s^{-1})$ is integrable for $t > 0$, the second integral can now be estimated by
	\begin{multline*}
		\int_{e^3 r}^R \frac{\Psi^n(r)}{\Psi^{n}(s)} \frac{\dd s}{s \log^{\mu + 1}(s^{-1})}
		\leq \int_{e^3 r}^R \left( \frac{\log (C_3 e^2 s^{-n})}{\log (C_3 r^{-n})} \right)^{\lambda} \frac{\dd s}{s \log^{\mu + 1}(s^{-1})} \\
		\leq  \left(\int_0^R \frac{\dd s}{s \log^{-\lambda} (C_3 e^2 s^{-n}) \log^{\mu + 1}(s^{-1})} \right) \log^{-\lambda} \frac{C_3}{r^{n}}
		\leq C_6 \log^{-\lambda} r^{-1}
	\end{multline*}
	where $C_6 = C_6(\Omega, \lambda, \mu, K, R)$.
	On the other hand, for the first integral, we may merely use the fact that $\Psi$ is increasing to estimate that $\Psi^n(s) \geq \Psi^n(r)$, which again combined with $\mu > \lambda$ yields that
	\begin{multline*}
		\int_r^{e^3 r} \frac{\Psi^n(r)}{\Psi^{n}(s)} \frac{\dd s}{s \log^{\mu + 1}(s^{-1})}
		\leq \int_r^{e^3 r} \frac{\dd s}{s \log^{\mu + 1}(s^{-1})}\\
		= \frac{1}{\mu} \left( \log^{-\mu}\frac{1}{e^3 r} - \log^{-\mu} \frac{1}{r} \right) \leq C_7 \log^{-\lambda} r^{-1}
	\end{multline*}
	where $C_7 = C_7(\mu, \lambda, R)$. The proof of the claimed estimate is hence complete.
\end{proof}

We then proceed to prove Theorem \ref{thm:exp_cont_modulus}. We again begin by recalling the statement.

\begin{customthm}{\ref{thm:exp_cont_modulus}}
	Let $\Omega \subset \R^n$ be a  domain, and let $f \in W^{1,n}_\loc(\Omega, \R^n)$ and   $Df \in \mathcal M_n (K, \Sigma)$ with 
	\begin{align*}
		\exp (\lambda K) \in L^1_\loc(\Omega) \qquad \text{and} \qquad \Sigma \log^{\mu} \left( e + \Sigma \right) \in L^1_\loc(\Omega),
	\end{align*}
	for some $\mu > \lambda > n+1$. Then $f$ has a continuous representative. 
	
	In particular, for all $x_0 \in \Omega$ and sufficiently small $r>0$, we have the following local modulus of continuity estimate:
	\[ 
		\omega_f (x_0, r) \le C \log^{-\alpha} (1/r)  \quad \textnormal{where } \alpha =\frac{\lambda-n-1}{n}\, . 
	\]
\end{customthm}

\begin{proof}
	Fix a ball $B = \B^n(x, R)$ such that $B$ is compactly contained in $\Omega$ and $R < R_0$, with $R_0$ given by Lemma \ref{lem:DfperK_decay}. Let $A \subset B$ be the set of all Lebesgue points of $f$ in $B$. We show first that the restriction of $f$ to $A \cap \B^n(x, R/(4e^3))$ is continuous. For this, let $y, z \in A \cap \B^n(x, R/(4e^3))$.
	
	We may select a two-sided sequence balls $B_i \subset B, i \in \Z$ in the following way: $B_0 = \B^n((y+z)/2, r_0)$ with $r_0 = \abs{y-z} \in (0, R/(2e^3))$, $B_i = \B^n(y, e^{-\abs{i}}r_0)$ for $i \in \Z_{>0}$, $B_i = \B^n(z, e^{-\abs{i}}r_0)$ for $i \in \Z_{<0}$. We denote the integral average of $f$ over $B_i$ by $f_{B_i} \in \R^n$; since $y$ and $z$ are Lebesgue points, we have 
	\begin{equation}\label{eq:Lebesgue_point_convs}
		\lim_{i \to \infty} f_{B_i} = f(y), \qquad \lim_{i \to -\infty} f_{B_i} = f(z).
	\end{equation}
	Moreover, since $\B^n(y, R/2)$ and $\B^n(z, R/2)$ and $\B^n((y+z)/2, R/2)$ are all contained in $B$ and $r_i < (R/2)/e^3$, Lemma \ref{lem:DfperK_decay} yields for every $i \in \Z$ that
	\begin{equation}\label{eq:DfperK_for_sequence}
		\int_{B_i} \frac{\abs{Df}^n}{K}
		\leq C_1 \log^{-\lambda} \frac{1}{r_i} = C_1 \left( \log \frac{1}{\abs{y-z}} + \abs{i} \right)^{-\lambda},
	\end{equation}
	with $C_1 = C_1(B, \lambda, \mu, K, \Sigma, R/2, \norm{Df}_{L^n(B)})$ independent of $y$, $z$, and $i$.
	
	We then estimate $\smallabs{f_{B_{i+1}} - f_{B_i}}$. We present the case $i \geq 0$, as the case $i < 0$ is similar but with $i$ and $i-1$ switched. Since $B_{i+1} \subset B_i$ and the radius of $B_{i+1}$ is $e^{-1}$ times the radius of $B_i$, we have by the Sobolev-Poincar\'e inequality that
	\begin{multline*}
		\abs{f_{B_{i+1}} - f_{B_i}} \leq \dashint_{B_{i+1}} \abs{f - f_{B_i}} \leq e^{-n} \dashint_{B_{i}} \abs{f - f_{B_i}}\\
		\leq C_2(n) r_i \left( \dashint_{B_{i}} \abs{Df}^{n-1} \right)^\frac{1}{n-1}
	\end{multline*}
	We then use H\"older's inequality to estimate that
	\begin{align*}
		r_i \left( \dashint_{B_{r_{i}}} \abs{Df}^{n-1} \right)^{\frac{1}{n-1}}
		&\leq r_i \left( \dashint_{B_{r_{i}}} \frac{\abs{Df}^{n}}{K} \right)^{\frac{1}{n}} \left( \dashint_{B_{r_{i}}} K^{n-1} \right)^{\frac{1}{n^2 - n}}\\
		&= \frac{1}{\sqrt[n]{\omega_n}} \left( \int_{B_{r_{i}}} \frac{\abs{Df}^{n}}{K} \right)^{\frac{1}{n}} \left( \dashint_{B_{r_{i}}} K^{n-1} \right)^{\frac{1}{n^2 - n}}.
	\end{align*}
	Applying the decay estimate \eqref{eq:DfperK_for_sequence}, we hence have that
	\[
	\abs{f_{B_{i+1}} - f_{B_i}} 
	\leq C_3 \left( \dashint_{B_{r_{i}}} K^{n-1} \right)^{\frac{1}{n^2 - n}} \left(\log \frac{1}{\abs{y-z}} + \abs{i}\right)^{-\frac{\lambda}{n}},
	\]
	where $C_3 = C_3(B, \lambda, \mu, K, \Sigma, \norm{Df}_{L^n(B)})$. We then estimate the average integral term. For this, we again define $\tilde{K} = \max(K, (n-2)\lambda^{-1})$ as in Lemma \ref{lem:triple_Jensen_K_estimate}, and use Jensen's inequality with the function $\tau \mapsto \exp (\lambda \tau^\frac{1}{n-1})$. This yields the estimate
	\begin{multline}\label{eq:K_decay}
		\left( \dashint_{B_{r_{i}}} K^{n-1} \right)^{\frac{1}{n^2 - n}}
		\leq \lambda^{-\frac{1}{n}} \log^{\frac{1}{n}} \left( \dashint_{B_{r_{i}}} \exp (\lambda \tilde{K}) \right)\\
		\leq \lambda^{-\frac{1}{n}} \log^{\frac{1}{n}} \left( \frac{\smallnorm{\exp(\lambda \tilde{K})}_{L^1(B)}}{\omega_n r_i^n} \right)\\
		\leq C_4 \log^\frac{1}{n} \frac{1}{r_i}
		= C_4 \left( \log \frac{1}{\abs{y-z}} + \abs{i} \right)^\frac{1}{n},
	\end{multline}
	where $C_4 = C_4(B, \lambda, K)$.
	
	Now, by \eqref{eq:Lebesgue_point_convs} and a telescopic sum argument, we obtain that
	\begin{equation}\label{eq:telescopic_sum_estimate}
		\abs{f(y) - f(z)} \leq \sum_{i = -\infty}^\infty \abs{f_{B_{i+1}} - f_{B_i}}
		\leq 2 \sum_{i = 0}^\infty C_5 \left(\log \frac{1}{\abs{y-z}} + i\right)^{\frac{1 - \lambda}{n}},
	\end{equation}
	with $C_5 = C_5(B, \lambda, \mu, K, \Sigma, \norm{Df}_{L^n(B)})$. We then denote $a = \log (1/\abs{y-z})$, noting that $a > 1$ since $\abs{y-z} < R < R_0 < e^{-1}$. Since we also assume that $\lambda > n+1$, we have that $i \mapsto (a + i)^{(1-\lambda)/n}$ is decreasing, and we may estimate
	\begin{multline*}
		\sum_{i = 0}^\infty \left(a + i\right)^{\frac{1 - \lambda}{n}}
		\leq a^{\frac{1 - \lambda}{n}} + \int_{0}^\infty (a + t)^{\frac{1 - \lambda}{n}} \dd t\\ 
		= a^{-\frac{\lambda - 1}{n}} + \frac{n}{\lambda - n - 1} a^{-\frac{\lambda - n - 1}{n}}
		\leq \frac{\lambda - 1}{\lambda - n - 1} a^{-\frac{\lambda - n - 1}{n}}
	\end{multline*}
	In conclusion,
	\begin{multline}\label{eq:modulus_of_continuity}
		\abs{f(y) - f(z)} \leq 2C_5 \sum_{i = 0}^\infty \left(a + i\right)^{\frac{1 - \lambda}{n}}\\
		\leq 2 C_5 \frac{\lambda - 1}{\lambda - n - 1} a^{-\frac{\lambda - n - 1}{n}} 
		= C_6 \log^{-\frac{\lambda - n - 1}{n}} \frac{1}{\abs{y-z}},
	\end{multline}
	with $C_6 = C_6(B, \lambda, \mu, K, \Sigma, \norm{Df}_{L^n(B)})$.
	
	We hence have obtained the desired modulus of continuity for all Lebesgue points $y, z \in A \cap \B^n(x, R/(4e^3))$. Now, if $y \in \B^n(x, R/(4e^3)) \setminus A$, we can then use the fact that $A$ has full measure in $\B^n(x, R/(4e^3))$ to select $y_j \in A \cap \B^n(x, R/(4e^3))$ such that $y_j \to y$ as $j \to \infty$. By \eqref{eq:modulus_of_continuity}, $(f(y_j))$ is a Cauchy sequence, and therefore convergent. We select $f(y) = \lim_{j \to \infty} f(y_j)$; doing this for all $y \in \B^n(x, R/(4e^3)) \setminus A$ only changes the values of $f$ in a set of measure zero, and doesn't change $f(y)$ in points $y$ where $f$ is continuous. Now, by passing to the limit, we see that \eqref{eq:modulus_of_continuity} applies to all $y, z \in \B^n(x, R/(4e^3))$. Hence, $f$ has a continuous representative with the desired modulus of continuity.
\end{proof}

Theorem \ref{thm:MFDvalue_cont} then follows as an immediate corollary of already proven results.

\begin{customthm}{\ref{thm:MFDvalue_cont}}
	Let $\Omega$ be a domain in $\R^n$. Suppose that a Sobolev mapping $f \in W_{\loc}^{1,n} (\Omega, \R^n)$ satisfies $Df \in \cM_n(K, \Sigma, y_0)$ with $K \colon \Omega \to [1, \infty)$, $\Sigma \colon \Omega \to [0, \infty)$ and $y_0 \in \R^n$. If
	\begin{align*}
		\exp (\lambda K) \in L^1_\loc(\Omega) \qquad \text{and} \qquad \Sigma \log^{\mu} \left( e + \Sigma \right) \in L^1_\loc(\Omega),
	\end{align*}
	for some $\mu > \lambda + n - 1 > 2n$, then $f$ has a continuous representative. 
	
	In particular, for all $x_0 \in \Omega$ and sufficiently small $r>0$, we have the following local modulus of continuity estimate:
	\[ 
		\omega_f (x_0, r) \le C \log^{-\alpha} (1/r)  \quad \textnormal{where } \alpha =\frac{\lambda-n-1}{n}\, . 
	\]
\end{customthm}
\begin{proof}
	Since $\Sigma \log^{\mu} \left( e + \Sigma \right) \in L^1_\loc(\Omega)$, and since $\exp(A\abs{f}^{n/(n-1)}) \in L^1_\loc(\Omega)$ for some $A > 0$ by Theorem \ref{thm:Zygmund_Sobolev_embedding}, we have  $\Sigma \abs{f}^n \log^{\mu - n + 1} \left( e + \Sigma \abs{f}^n \right) \in L^1_\loc(\Omega)$ by Lemma \ref{lem:product_int_lemma}, where $\mu - n + 1 > \lambda > n+1$. Hence, the claim follows by applying Theorem \ref{thm:exp_cont_modulus}.
\end{proof}

\section{Counterexamples based on cusps}

In this section, we consider our first type of counterexample, which yields Theorems \ref{thm:discontinuity_K_Lp} and \ref{thm:cusp_ex_version_2}. Our construction will be in a planar disk $\D(r_0)$ with center at the origin and radius $r_0$. Our constructed mapping $f \colon \D(r_0) \to \R^2$ has a first coordinate function of $-\log \log \abs{z}^{-1}$, which is well defined in $\D(r_0) \setminus \{0\}$ as long as $r_0$ is small enough. We split the disk $\D(r_0)$ into two regions $\D(r_0) = A \cup B$ with different definitions of the second coordinate function, where in $A$ we aim to have $J_f(x) \geq 0$ with $\abs{Df(x)}^2 \leq K J_f(x)$, and in $B$ we try to obtain $\abs{Df(x)}^2 + K \abs{J_f(x)} \leq \Sigma$. The region $B$ will form a cusp at the origin.

\subsection{The two regions}

Let $\Omega = \D(r_0)$, where we assume that $r_0 \leq e^{-e}$. We begin by assuming that $\gamma$ is an absolutely continuous increasing function $\gamma \colon [0, r_0) \to [0, 1)$ such that $\gamma(0) = 0$. We specify $\gamma$ later in the text, as we use different choices of $\gamma$ to prove different theorems. We will use polar coordinates $(r, \theta)$ on the domain side in $\Omega$, where $\theta \in (-\pi, \pi]$.

The regions $A, B \subset \Omega$ will consist of two sub-regions $A = A_1 \cup A_2$ and $B = B_1 \cup B_2$ each. We let $B_1$ be the cusp-like region of $\Omega$ bounded by the curves $\theta = \gamma(r)$ and $\theta = - \gamma(r)$. Similarly, we let $A_1$ be the region bounded by the curves $\theta = \gamma(r)$ and $\theta = \pi - \gamma(r)$. The region $B_2$ is the reflection $-B_1$ of $B_1$ across the origin, and similarly $A_2 = -A_1$. See Figure \ref{fig:domains_A_and_B} for an illustration. 

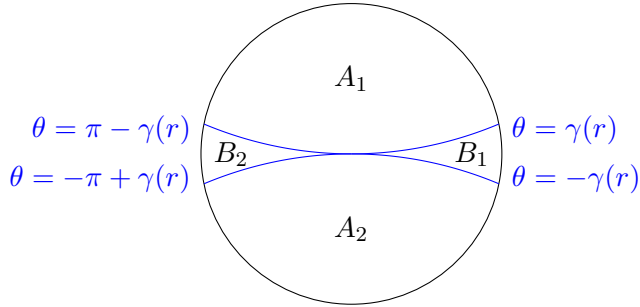
\begin{figure}[h t p]
	\centering
	\begin{tikzpicture}
		\draw (0,0) circle(2.0);
		
		\draw [blue,  domain=0:1, samples=40] 
		plot ({2*\x*(sin((pi/2+0.2*\x) r))}, {2*(\x)*(cos((pi/2+0.2*\x) r))});
		\draw [blue,  domain=0:1, samples=40] 
		plot ({2*\x*(sin((pi/2-(0.2*\x)) r))}, {2*(\x)*(cos((pi/2-(0.2*\x)) r))});

		\draw [blue,  domain=0:1, samples=40] 
		plot ({2*\x*(sin((3*pi/2+0.2*\x) r))}, {2*(\x)*(cos((3*pi/2+0.2*\x) r))});
		\draw [blue,  domain=0:1, samples=40] 
		plot ({2*\x*(sin((3*pi/2-(0.2*\x)) r))}, {2*(\x)*(cos((3*pi/2-(0.2*\x)) r))});
		
		\node at (0,1) {$A_1$};
		\node at (1.6,0) {$B_1$};
		\node at (-1.6,0) {$B_2$};
		\node at (0,-1) {$A_2$};
		
		\node at (2, 0)[blue, anchor=south west] {$\theta = \gamma(r)$};
		\node at (2, 0)[blue, anchor=north west] {$\theta = -\gamma(r)$};
		\node at (-2, 0)[blue, anchor=south east] {$\theta = \pi -\gamma(r)$};
		\node at (-2, 0)[blue, anchor=north east] {$\theta = -\pi+\gamma(r)$};
	\end{tikzpicture}
	\caption{The regions $A_1$, $A_2$, $B_1$ and $B_2$.}\label{fig:domains_A_and_B}
\end{figure}

\subsection{The function $f$ in the region $A$}

We first define $f$ in the region $A_1$. There, using polar coordinates on the domain side and Cartesian coordinates on the target side, we have
\begin{equation}\label{eq:f_def_A}
	f(r, \theta) = (- \log \log r^{-1}, h(r) \theta)
\end{equation}
for some absolutely continuous increasing function $h \colon [0, \infty) \to [0, \infty)$. We again specify $h$ later. 

Hence, we obtain a matrix of derivatives
\[
\begin{bmatrix}
	\partial_r f_1 & r^{-1} \partial_\theta f_1\\
	\partial_r f_2 & r^{-1} \partial_\theta f_2
\end{bmatrix}
=
\begin{bmatrix}
	r^{-1} \log^{-1} r^{-1} & 0\\
	h'(r) \theta & r^{-1} h(r)
\end{bmatrix}.
\]
In particular,
\begin{equation}\label{eq:f_A_Df_estimate}
	\abs{Df(r, \theta)}^2 \leq \frac{1}{r^2 \log^2 r^{-1}} + [h'(r)\theta]^2 + \frac{h^2(r)}{r^2}
\end{equation}
and
\begin{equation}\label{eq:f_A_Jf_estimate}
	J_f(r, \theta) = \frac{h(r)}{r^2 \log r^{-1}} \geq 0.
\end{equation}
We then simply pick $K = \abs{Df}^2 / J_f$ and $\Sigma \equiv 0$.

On $A_2$, we define $f(z) = f(-z)$. Since $z \mapsto -z$ is an orientation-preserving isometry in the plane, it follows that \eqref{eq:f_A_Df_estimate} and \eqref{eq:f_A_Jf_estimate} remain true in $A_2$.

\subsection{The function $f$ in the region $B$}

We wish that our function $f$ is continuous outside the origin. Hence, our boundary values in $B_1$ must match the ones given by $A_1$ and $A_2$. For this, we define the second coordinate of $f$ as a linear interpolation of these boundary values. That is, we define in $B_1$ that
\begin{equation}\label{eq:f_def_B}
	f(r, \theta) = \left(- \log \log r^{-1}, h(r)\left(\frac{\pi + 2\theta - \pi \theta/\gamma(r)}{2}\right) \right).
\end{equation}
Indeed, in the cases $\theta = \gamma(r)$ and $\theta = -\gamma(r)$, the second coordinate has the correct boundary values of $h(r) \gamma(r)$ and $h(r) (\pi - \gamma(r))$, respectively.

The derivatives of the first coordinate of $f$ remain unchanged from domain $A$. For the other terms in the matrix of derivatives, we first get
\begin{equation}\label{eq:f_B_second_coord_dr}
	\partial_r f_2 = \left(\frac{\pi}{2} + \theta\right) h'(r) - \frac{\pi \theta}{2} \frac{d}{dr} \left(\frac{h(r)}{\gamma(r)} \right).
\end{equation}
Then, by $\gamma(r) < 1$, we get
\begin{equation}\label{eq:f_B_second_coord_dtheta}
	r^{-1} \partial_\theta f_2 = \frac{h(r)}{r}  \left(1 - \frac{\pi}{2\gamma(r)}\right)
	< - \left(\frac{\pi}{2} - 1 \right) \frac{h(r)}{r\gamma(r)}
\end{equation}
In particular, we have that $r^{-1} \partial_\theta f_2 < 0$, and consequently $J_f < 0$ in $B_1$. Hence, in $B_1$, we select $K = - \abs{Df}^2/J_f \geq 1$, and $\Sigma = 2\abs{Df}^2$.

Similarly as for $A_2$, we may define $f$ in $B_2$ by $f(z) = f(-z)$, and all our considerations will also apply to $B_2$.

\subsection{Fixing the parameters}\label{subsect:cusp_proofs}

We have now outlined the construction, but have left the functions $h$ and $\gamma$ undetermined. The theorems we wish to prove follow with different choices of $h$ and $\gamma$. 

Throughout the rest of this paper, given two functions $f, g \colon X \to \R$, we use the notation $f \lesssim g$ if there exists a constant $C > 0$ such that $f \leq C g$. We also denote $f \approx g$ if $f \lesssim g \lesssim f$. Several of the uses of these symbols are based on the elementary fact that if $f, g \colon [a, \infty) \to (0, \infty)$, $a \in \R$, are continuous and $\limsup_{t \to \infty} f(t)/g(t) < \infty$, then $f \lesssim g$.

We now recall the statement of Theorem \ref{thm:cusp_ex_version_2}, and then give its proof.

\begin{customthm}{\ref{thm:cusp_ex_version_2}}
	Let $p, q \in (1, \infty)$. If $p^{-1} + q^{-1} \geq 1$, then there exists a domain $\Omega \subset \mathbb R^2$ and a Sobolev map $f\in W^{1,2} (\Omega, \mathbb R^2)$ such that $0\in \Omega$, $f \in C (\Omega \setminus \{0\}, \mathbb R^2)$, $\lim_{x \to \infty} \abs{f(x)}=\infty$, and $Df \in \mathcal M_2(K, \Sigma)$ with
	\[ 
		K \in L^p  (\Omega) \qquad \text{and} \qquad \frac{\Sigma}{K} \in L^q (\Omega).
	\]
\end{customthm}

\begin{proof}
	Let $p, q \in (1, \infty)$, and let $\eps > 0$. We select $r_0 = e^{-e}$ and 
	\[
		h(r) = r^{2p^{-1}}, \qquad \gamma(r) = \log^{-\eps} r^{-1}.
	\]
	Indeed, when $r < e^{-e}$, we have $0 \leq \gamma(r) < e^{-\eps} < 1$.
	
	In $A_1$, we have by \eqref{eq:f_A_Df_estimate} that
	\[
	\abs{Df(r, \theta)}^2 \leq \frac{1}{r^2 \log^2 r^{-1}} + \left(\frac{4\pi^2}{p^2} + 1\right) r^{-2 + 2p^{-1}} \lesssim \frac{1}{r^2 \log^2 r^{-1}}.
	\]
	Hence, $\abs{Df} \in L^2(A)$. By also referring to \eqref{eq:f_A_Jf_estimate}, we have in $A_1$ the estimate
	\[
	K(r, \theta) 
	= \frac{\abs{Df(r, \theta)}^2}{J_f(r, \theta)}
	\lesssim \frac{1}{r^{2p^{-1}} \log r^{-1}}.
	\]
	We hence estimate that
	\begin{align*}
		\int_A K^p 
		&= 2 \int_0^{e^{-e}} \int_{\gamma(r)}^{\pi - \gamma(r)} K^p(r, \theta) r \dd\theta \dd r\\
		&\lesssim \int_0^{e^{-e}} r^{-1} \log^{-p} r^{-1} \dd r < \infty,
	\end{align*}
	showing that $K \in L^p(A)$.
	
	We then consider points $(r, \theta)$ in $B_1$. By \eqref{eq:f_B_second_coord_dr} and $\abs{\theta} \leq \gamma(r) < 1 \leq \pi$, we have
	\begin{align*}
		\abs{\partial_r f_2(r, \theta)}^2 &= \abs{\frac{2\theta + \pi}{p} r^{2p^{-1} - 1} +\frac{\pi \theta \log^{\eps} r^{-1}}{2} r^{2p^{-1} - 1}\left(\frac{2}{p} - \frac{\eps}{\log r^{-1}} \right)}^2\\
		&\lesssim r^{4p^{-1} - 2} \left(1 + \log^{-2 + 2\eps} r^{-1}\right).
	\end{align*}
	Furthermore, by \eqref{eq:f_B_second_coord_dtheta},
	\begin{align*}
		\abs{r^{-1} \partial_\theta f_2(r, \theta)}^2 &\lesssim r^{4p^{-1} - 2} \log^{2\eps} r^{-1}.
	\end{align*}
	The exponent $4p^{-1} - 2$ in the above bounds is greater than $-2$. Hence, we have the overall estimate
	\begin{equation}\label{eq:cusp_ex_polylog_B_estimated_Df}
		\abs{Df(r, \theta)}^2 \approx \frac{1}{r^2 \log^2 r^{-1}}
	\end{equation}
	whenever $(r, \theta) \in B_1$. In particular, we have $\abs{Df} \in L^2(B_1)$, and consequently $\abs{Df} \in L^2(\Omega)$. Moreover, we have
	\[
	-J_f(r, \theta) = \left(\frac{\pi}{2} \log^\eps r^{-1} - 1 \right) r^{2p^{-1} - 2} \log^{-1} r^{-1} ,
	\]
	so by $\log^\eps r^{-1} > 1$ we get the two-sided estimate
	\begin{equation}\label{eq:cusp_ex_polylog_B_estimated_Jf}
	-J_f(r, \theta) \approx r^{2p^{-1} - 2} \log^{\eps -1} r^{-1}.
	\end{equation}
	Combined with \eqref{eq:cusp_ex_polylog_B_estimated_Df}, this yields
	\[
		\frac{\abs{Df(r, \theta)}^2}{-J_f(r, \theta)} = K(r, \theta) \approx \frac{1}{r^{2p^{-1}}\log^{1+\eps} r^{-1}}.
	\]
	Since $p(2p^{-1}) = 2$ and $p(1+\eps) > 1$, we see that $K \in L^p(B)$, and hence $K \in L^p(\Omega)$.
	
	It hence remains to consider the integral $(\Sigma/K)^q$ over $B$. Since we chose $\Sigma = \abs{Df}^2$ and $K = \abs{Df}^2 / (-J_f(x))$, we have $\Sigma/K = -J_f(x)$. Hence, by \eqref{eq:cusp_ex_polylog_B_estimated_Jf}, we have $\Sigma/K \lesssim r^{2p^{-1} - 2} \log^{\eps -1} r^{-1}$. We note that since $B$ is a cusp, this majorant in fact has a better degree of integrability over $B$ than it has over $\Omega$. In particular, we may estimate that
	\begin{align*}
		\int_{B} \frac{\Sigma^q}{K^q}
		&\lesssim \int_0^{e^{-e}} \int_{-\gamma(r)}^{\gamma(r)} \frac{r \dd \theta \dd r}{r^{2q - 2p^{-1}q} \log^{q - q\eps} r^{-1}}\\
		&\leq 2 \int_0^{e^{-e}} \frac{\dd r}{r^{2q - 2p^{-1}q - 1} \log^{q - (q - 1)\eps} r^{-1}}.
	\end{align*}
	For integrability, we require $2q - 2p^{-1}q - 1 \leq 1$, which is equivalent to $q^{-1} \geq 1 - p^{-1}$. Moreover, in the extremal case $q^{-1} + p^{-1} = 1$, we also require $q - (q-1)\eps > 1$, which is equivalent to $\eps < 1$. Hence, any choice of $\eps \in (0, 1)$ will give us the desired example.
\end{proof}

Our next result is the version of this example with the highest degree of integrability for $\Sigma$. This is by a different choice of $h$ and $\gamma$, and hence this gain in the regularity of $\Sigma$ comes at a cost in the regularity of $\Sigma/K$.

\begin{thm}\label{lem:is_this_also_geq}
	Let $p, s \in (1, \infty)$. If $(p+1)^{-1} + s^{-1} \geq 1$, then there exists a domain $\Omega \subset \mathbb R^2$ and a Sobolev map $f\in W^{1,2} (\Omega, \mathbb R^2)$ such that $0\in \Omega$, $f \in C (\Omega \setminus \{0\}, \mathbb R^2)$, $\lim_{x \to \infty} \abs{f(x)}=\infty$, and $Df \in \mathcal M_2(K, \Sigma)$ with
	\[  
	K \in L^p  (\Omega) \qquad \text{and} \qquad  \Sigma \in L^s (\Omega).
	\]
\end{thm}
\begin{proof}
	Let $p, q \in (1, \infty)$, and let $\eps > 0$. This time we choose
	\[
		h(r) = r^{2p^{-1}}, \qquad \gamma(r) = r^{2p^{-1}} \log r^{-1}.
	\]
	We may select an $r_0 \leq e^{-e}$ such that $\gamma$ is increasing on $[0, r_0]$ and $\gamma(r_0) < 1$.
	
	The verification that $K \in L^p(A)$ is unchanged from the previous lemma. The difference arises when applying \eqref{eq:f_B_second_coord_dr} and \eqref{eq:f_B_second_coord_dtheta}. Indeed, since $\abs{\theta} \leq r^{2p^{-1}} \log r^{-1} \leq 1$, we obtain.
	\begin{align*}
		\abs{\partial_r f_2(r, \theta)}^2 &= \abs{\frac{2\theta + \pi}{p} r^{2p^{-1} - 1} +\frac{\pi \theta}{2} r^{-1} \log^{-2} r^{-1}}^2\\
		&\lesssim r^{4p^{-1} - 2} + r^{- 2} \log^{-4} r^{-1},
	\end{align*}
	and
	\begin{align*}
		\abs{r^{-1} \partial_\theta f_2(r, \theta)}^2 
		&\lesssim r^{-2} \log^{-2} r^{-1}.
	\end{align*}
	In all of the previously computed terms, either the exponent of $r$ is greater than $-2$, or the exponent of $r$ is $-2$ and the exponent of the logarithm is at most $-2$. Hence, we still have \eqref{eq:cusp_ex_polylog_B_estimated_Df} unchanged. For $J_f$, we compute similarly as in the last lemma, and instead get
	\begin{equation}\label{eq:cusp_ex_poly-poly_B_estimated_Jf}
		-J_f(r, \theta) \approx r^{- 2} \log^{-2} r^{-1}
	\end{equation}
	when $(r, \theta) \in B$. In particular, $K = \abs{Df}^2/(-J_f) \in L^\infty(B)$.
	
	It remains to estimate the integral of $\Sigma^s = (2\abs{Df})^s$ over $B$. Computing similarly as in the previous lemma, we get
	\begin{align*}
		\int_{B} \Sigma^q
		&\lesssim \int_0^{r_0} \int_{-\gamma(r)}^{\gamma(r)} \frac{r \dd \theta \dd r}{r^{2s} \log^{2s} r^{-1}}\\
		&\leq 2 \int_0^{r_0} \frac{\dd r}{r^{2s - 2p^{-1} - 1} \log^{2s - 1} r^{-1}}.
	\end{align*}
	For this to converge, since the exponent of the logarithm satisfies $2s - 1 > 1$ due to $s > 1$, we only require $2s - 2p^{-1} - 1 \leq 1$. Rearranging yields $s \leq 1 + p^{-1} = (p+1)^*$, where $(p+1)^*$ is the H\"older conjugate of $p+1$. In particular, this is equivalent with $(p+1)^{-1} + s^{-1} \geq 1$.
\end{proof}

The remaining result which relies on this example type is Theorem \ref{thm:discontinuity_K_Lp}. This is achieved by selecting both $h$ and $\gamma$ to be powers of logarithms, with a suitable choice of corresponding exponents.

\begin{customthm}{\ref{thm:discontinuity_K_Lp}}
	For every $\mu \in (0, 2)$, there exist a domain $\Omega \subset \mathbb R^2$ and a Sobolev map $f\in W^{1,2} (\Omega, \mathbb R^2)$ such that $0\in \Omega$, $f \in C (\Omega \setminus \{0\}, \mathbb R^2)$, $\lim_{x \to \infty} \abs{f(x)}=\infty$, and $Df \in \mathcal M_2(K, \Sigma)$ with
	\[  
		\exp (\lambda K )\in L^{1}(\Omega)  
		\qquad \text{and} \qquad   
		\Sigma \log^\mu (e+\Sigma) \in L^1(\Omega) \, 
	\]
	for every $\lambda > 0$.
\end{customthm}
\begin{proof}
	Let $\lambda \in (0, \infty)$. We may assume $\mu > 1$, as an example for a given $\mu$ also works for all smaller $\mu$. We choose
	\[
		h(r) = \log^{-\nu} r^{-1}, \qquad \gamma(r) = \log^{1-\nu} r^{-1},
	\]
	Where $\nu \in (\mu, 2)$. Since $\nu > \mu > 1$ by assumption, $\gamma$ is increasing, and we may hence choose $r_0 = e^{-e}$ as $\gamma(e^{-e}) = e^{1-\nu} < 1$.
	
	In $A_1$, \eqref{eq:f_A_Df_estimate} yields due to $\nu > 1$ that
	\[
		\abs{Df(r, \theta)}^2 \leq \frac{1}{r^2 \log^2 r^{-1}} + \frac{\nu}{r^2 \log^{2\nu + 2} r^{-1}} + \frac{1}{r^2 \log^{2\nu} r^{-1}} \lesssim \frac{1}{r^2 \log^2 r^{-1}}.
	\]
	Hence, clearly $\abs{Df} \in L^2(A)$. Moreover, $J_f = r^{-2} \log^{-1-\nu} r^{-1}$, so
	\[
		K(r, \theta) \leq C \log^{\nu - 1} r^{-1}
	\]
	for some $C > 0$. The exponential integrals of $K$ are all finite by the estimate
	\begin{align*}
		\int_A \exp(\lambda K)
		&\leq 2\pi \int_0^{e^{-e}} \exp(C \lambda \log^{\nu - 1} r^{-1}) r \dd r\\
		&= 2\pi \int_0^{e^{-e}} r^{1 - C \lambda \log^{\nu - 2} r^{-1}} \dd r < \infty,
	\end{align*}
	since $\lim_{r \to 0^+} C \lambda \log^{\nu - 2} r^{-1} = 0$ due to $\nu < 2$.
	
	In $B_1$, \eqref{eq:f_B_second_coord_dr} and \eqref{eq:f_B_second_coord_dtheta} combined with $\abs{\theta} \leq \log^{1-\nu} r^{-1} \leq 1 \leq \pi$ result in
	\begin{align*}
		\abs{\partial_r f_2(r, \theta)}^2 &= \abs{\frac{2\nu\theta + \pi\nu}{2} r^{-1} \log^{-\nu - 1} r^{-1} +\frac{\pi \theta}{2} r^{-1} \log^{-2} r^{-1}}^2\\
		&\lesssim r^{- 2}  \left( \log^{-2-2\nu} r^{-1} + \log^{-4} r^{-1} \right)
	\end{align*}
	and
	\begin{align*}
		\abs{r^{-1} \partial_\theta f_2(r, \theta)}^2 
		&\lesssim  r^{- 2} \log^{-2} r^{-1}.
	\end{align*}
	As $\log^{t} r^{-1}$ is increasing with respect to $t$ when $r < e^{-e}$, we again have 
	\[
		\abs{Df(r, \theta)}^2 \leq \frac{C'}{r^{2} \log^{2} r^{-1}}
	\] 
	in $B_1$ for some $C' > 0$. For the Jacobian, we instead get
	\begin{equation}\label{eq:cusp_ex_log-log_B_estimated_Jf}
		-J_f(r, \theta) \approx r^{- 2} \log^{-2} r^{-1}.
	\end{equation}
	In particular, our choice $K = \abs{Df}^2/(-J_f)$ is in $L^\infty(B)$, concluding exponential integrability of $K$ in all of $\Omega$ for all choices of $\lambda$.
	
	The last step is to estimate the integral of $\Sigma \log^\mu(e + \Sigma)$ over $B_1$, where $\Sigma = 2\abs{Df}^2$. We estimate using $(e + ab) \leq (e + a)(e + b)$ for $a,b \geq 0$ that
	\[
		\log (e + \Sigma) \leq \log\left(e + \frac{2C'}{r^{2} \log^{2} r^{-1}}\right) \leq \log\left(e + \frac{2C'}{\log^2 r^{-2}}\right) + 2 \log(e + r^{-1}),
	\]
	and hence, as $\gamma(r) = \log^{1 - \nu} r^{-1}$, we get
	\begin{align*}
	 	\int_{B} \Sigma \log^{\mu}(e + \Sigma)
	 	&= 2 \int_0^{e^{-e}} \int_{-\gamma(r)}^{\gamma(r)} \Sigma(r, \theta) \log^{\mu}(e + \Sigma(r, \theta)) r \dd \theta \dd r\\
	 	&\lesssim \int_0^{e^{-e}} \frac{\log^\mu (e + 2C'\log^{-2} r^{-1}) + \log^\mu(e + r^{-1})}{r \log^{2 - (1 - \nu)} r^{-1}} \dd r.
	\end{align*}
 	When $r \to 0$, we have $\log^{-2} r^{-1} \to 0$. Hence, for small $r$, the largest term in the numerator is $\log^\mu(e + r^{-1})$. Since $r^{-1} > e^{e}$, we have $r^{-2} - r^{-1} - e \geq 0$. Hence, we may estimate
 	\[
 		\frac{\log^\mu(e + r^{-1})}{r \log^{2 - (1 - \nu)} r^{-1}} \leq \frac{2^\mu}{r \log^{1 + (\nu - \mu)} r^{-1}},
 	\]
 	which is integrable over $[0, e^{-e}]$ due to our assumption $\nu > \mu$. Thus, $\Sigma \log^\mu(e + \Sigma) \in L^1(B)$, and consequently $\Sigma \log^\mu(e + \Sigma) \in L^1(\Omega)$.
\end{proof}

\section{Counterexamples based on spirals}

In this section, we construct a counterexample built around the case $\Sigma \in L^\infty(\Omega)$, which will give us Theorem \ref{thm:discontinuity_bounded_sigma}.  Furthermore, if $K \in L^p_\loc(\Omega)$ with $p \in [1, 2]$, then this counterexample also yields an alternate proof of Theorem~\ref{thm:cusp_ex_version_2}.  In exchange for failing when $p > 2$, this alternate counterexample has a better optimal integrability for $\Sigma$ when $p < \sqrt{2}$; that is, it improves Theorem~\ref{lem:is_this_also_geq} for such values of $p$. Moreover, this improved integrability of $\Sigma$ is achieved simultaneously with the optimal integrability of $\Sigma/K$, whereas the construction of Theorem \ref{thm:cusp_ex_version_2} involves a trade-off between the integrabilities of $\Sigma$ and $\Sigma/K$. 
\begin{thm}\label{thm:spiral_ex_version_2}
	Suppose that $p \in [1, 2]$, $q \in [1, \infty]$, and $p^{-1} + q^{-1} \geq 1$. Then there exist a domain $\Omega \subset \mathbb R^2$ and a Sobolev map $f\in W^{1,2} (\Omega, \mathbb R^2)$ such that $0\in \Omega$, $f \in C (\Omega \setminus \{0\}, \mathbb R^2)$, $\lim_{x \to \infty} \abs{f(x)}=\infty$, and $Df \in \mathcal M_2(K, \Sigma)$ with
	\[  
		K \in L^p  (\Omega), \qquad \frac{\Sigma}{K} \in L^q (\Omega), \qquad \text{and} \qquad \Sigma \in L^\frac{q}{2}(\Omega).
	\]
\end{thm}
 
We again construct our example in a planar region $\Omega \subset \R^2$ with a point of discontinuity at the origin, and we retain our strategy from the previous section of splitting $\Omega$ into two regions $A$ and $B$, where $\abs{Df}^2 \leq K J_f$ in $A$ and $\abs{Df}^2 + K \abs{J_f} \leq \Sigma$ in $B$. Notably, when $\Sigma$ is bounded from above by a constant, $f$ ends up being Lipschitz under the path length metric in $B$. Hence, if we wish that $f$ escapes to infinity along $B$, the region $B$ must somehow be infinitely long. This pushes us towards a construction where $A$ and $B$ are two interlocking infinitely long spirals centered at the origin.

\subsection{Preliminaries: Lambert's $W$-function}

We begin by recalling a special function that is of great use to us in our construction. Namely, \emph{Lambert's $W$-function} is the inverse function $W = \psi^{-1}$ of the function $\psi(t) = t e^t$. The $W$-function has two branches on the real line. In this paper, we assume $W$ to be the positive branch:
\[
	W \colon [-e^{-1}, \infty) \to [-1, \infty), \quad W(t) e^{W(t)} = t.
\]
We collect into the following lemma the elementary properties of the $W$-function that we use. For a general reference on the $W$-function, see e.g.\ \cite{CorlessEtAl_Lambert-W}.
\begin{lemma}\label{lem:W_props}
	The $W$-function satisfies the following.
	\begin{enumerate}
		\item $W$ is strictly increasing on $[-e^{-1}, \infty)$.
		\item $W(0) = 0$, and hence $W(t) > 0$ if $t > 0$.
		\item We have $W(t \log t) = \log t$ if $t \geq e^{-1}$.
		\item The derivative of $W$ is given on $(-e^{-1}, \infty)$ by
		\[
		W'(t) = \frac{W(t)}{t(1 + W(t))} = \frac{1}{t + e^{W(t)}}.
		\]
	\end{enumerate}
\end{lemma}

\subsection{Construction}

We define two spirals in polar coordinates. The first one is the spiral $r = g(\theta)$, where
\[
	g(\theta) = \frac{1}{\theta \log \theta}, \qquad \theta \in [\theta_0, \infty),
\]
where $\theta_0 \geq 2\pi$ is some starting angle. The second one is given by $r = h(\theta)$, where
\[
	h(\theta) = \frac{g(\theta) + g(\theta + 2\pi)}{2}, \qquad \theta \in [\theta_0, \infty);
\]
that is, the spiral $r = h(\theta)$ lies exactly halfway between the successive points where the spiral $r = g(\theta)$ meets a specific ray from the origin. We define our domain $\Omega \subset \C$ by
\[
	\Omega = \{r e^{i\theta} : 0 \leq r < g(\theta), \theta \in (\theta_0, \theta_0 + 2\pi] \}.
\]
See Figure \ref{fig:double_spiral} for an illustration of $\Omega$ and the spirals.  
\begin{figure}[h]
	\centering
	\begin{tikzpicture}[scale=1.4]
		\draw [fill=gray!30] 
		plot[smooth, domain=2*pi:4*pi, samples=100]  ({(23/(\x*ln(\x)))*(cos((\x) r))}, {(23/(\x*ln(\x)))*(sin((\x) r))});
		
		\draw [blue,  domain=2*pi:10*pi, samples=500] 
		plot ({(23/(\x*ln(\x)))*(cos((\x) r))}, {(23/(\x*ln(\x)))*(sin((\x) r))});
		\draw [blue, dotted, domain=10*pi:11.5*pi, samples=50] 
		plot ({(23/(\x*ln(\x)))*(cos((\x) r))}, {(23/(\x*ln(\x)))*(sin((\x) r))});

		\draw [red,  domain=2*pi:10*pi, samples=500] 
		plot ({(((23/(\x*ln(\x)))+(23/((\x+2*pi)*(ln(\x+2*pi)))))/2)*(cos((\x) r))}, {(((23/(\x*ln(\x)))+(23/((\x+2*pi)*(ln(\x+2*pi)))))/2)*(sin((\x) r))});
		\draw [red, dotted, domain=10*pi:11.5*pi, samples=50] 
		plot ({(((23/(\x*ln(\x)))+(23/((\x+2*pi)*(ln(\x+2*pi)))))/2)*(cos((\x) r))}, {(((23/(\x*ln(\x)))+(23/((\x+2*pi)*(ln(\x+2*pi)))))/2)*(sin((\x) r))}); 
		
		\node at (1.8,0)[red, anchor=north] {$r=h(\theta)$};
		\node at (2,0.3)[blue, anchor=west] {$r=g(\theta)$};
		
	\end{tikzpicture}
	\caption{The two spirals $r = g(\theta)$ and $r = h(\theta)$, with the domain $\Omega$ highlighted in gray.}\label{fig:double_spiral}
\end{figure}
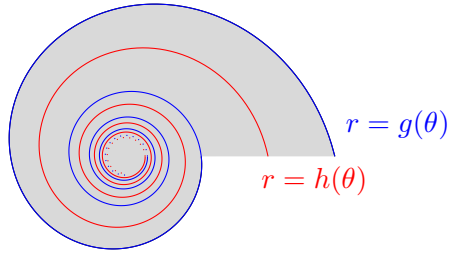

We parametrize $\Omega$ in the following way: let
\[
U = \{ (r, \theta) \in \mathbb{R}^2 : \theta \geq \theta_0, g(\theta+2\pi) \leq r < g(\theta)\},
\]
in which case the map $(r, \theta) \mapsto r e^{i\theta}$ maps $U$ bijectively to $\Omega \setminus \{0\}$. Let $\alpha \in (0, 1]$. We define $f \colon \Omega \setminus \{0\} \to \C$ on the two regions between the spirals $r = g(\theta)$ and $r = h(\theta)$ in terms of polar coordinates $(r, \theta) \in U$ : when $h(\theta) \leq r < g(\theta)$, we define 
\[
	f(r, \theta) = \varphi(r) - i \log \log \theta,
\]
where $\varphi(r) \colon [0, r_0] \to \R$ is an increasing absolutely continuous function to be fixed later, with $r_0 > \theta_0^{-1} \log^{-1} \theta_0$. In the other region where $g(\theta + 2\pi) \leq r < h(\theta)$, we instead define
\[
	f(r, \theta) = \varphi(r) - i \log W \left( \frac{1}{2r - g(\theta + 2\pi)}\right).
\]
This defines $f$ on all of $\Omega \setminus \{0\}$.

We briefly verify that $f$ is indeed continuous on $\Omega \setminus \{0\}$. If $r = h(\theta)$, then
\begin{multline*}
	\log W \left( \frac{1}{2r - g(\theta + 2\pi)}\right)
	= \log W \left( 
		\frac{1}{2h(\theta) - g(\theta + 2\pi)}\right)\\
	= \log W \left( \frac{1}{g(\theta)}\right)
	= \log W(\theta \log \theta)
	= \log \log \theta,
\end{multline*}
which verifies that $f$ is continuous on the spiral $r = h(\theta)$. On the other hand, if $r = g(\theta+2\pi)$, then
\begin{multline*}
	\log W \left( \frac{1}{2r - g(\theta + 2\pi)}\right)
	= \log W \left( 
	\frac{1}{g(\theta + 2\pi)}\right)\\
	= \log W((\theta + 2\pi) \log (\theta+ 2\pi))
	= \log \log (\theta + 2\pi),
\end{multline*}
which verifies continuity of $f$ on the spiral $r = g(\theta)$. Hence, $f$ is continuous on $\Omega \setminus \{0\}$.

\subsection{The first region}

We then compute $\abs{Df}$ and $J_f$ in the region $B \subset \Omega$ where $h(\theta) < r < g(\theta)$ in terms of our polar coordinate parametrization. In polar coordinates, the derivative matrix of $f$ becomes
\[
\begin{bmatrix}
	\partial_r \Re(f) & r^{-1} \partial_\theta \Re(f)\\
	\partial_r \Im(f) & r^{-1} \partial_\theta \Im(f)
\end{bmatrix}
=
\begin{bmatrix}
	\varphi'(r) & 0\\
	0 & -r^{-1} \theta^{-1} \log^{-1} \theta
\end{bmatrix}
=
\begin{bmatrix}
	\varphi'(r) & 0\\
	0 & -r^{-1} g(\theta)
\end{bmatrix}.
\]
Since we moreover have $r \geq h(\theta) = (g(\theta) + g(\theta + 2\pi))/2 \geq g(\theta)/2$, we obtain the upper bound
\[
	r^{-1} g(\theta) \leq 2.
\]
Hence, we have the estimate
\[
	\abs{Df(r, \theta)}^2 \leq 4 + (\varphi'(r))^2.
\]
Note especially that $\abs{Df}$ is bounded in $B$ when $\varphi$ is Lipschitz. Moreover, we have $\abs{Df} \in L^2(B)$ as long as $r (\varphi'(r))^2 \in L^1([0, r_0])$.

On the other hand, $J_f(r, \theta) = -r \varphi'(r) g(\theta)$, which is negative since $\varphi(r)$ is increasing. Furthermore, $-J_f(r, \theta)$ is bounded from above by $2 \varphi'(r)$.  Hence, in order to achieve the desired condition $\abs{Df}^2 + K \abs{J_f} \leq \Sigma$, we arrive at the following valid choices for $\Sigma$ and $K$:
\begin{equation}\label{eq:spiral_first_region}
	\Sigma(r, \theta) = 6 + 3(\varphi'(r))^2, \qquad \qquad K(r, \theta) = \max (\varphi'(r), 1).
\end{equation}

\subsection{The second region}

Next, we consider the region $A \subset \Omega$ where we have $g(\theta + 2\pi) \leq r < h(\theta)$ in terms of our polar coordinate parametrization. We still have $\partial_r \Re(f) = \varphi'(r)$ and $\partial_\theta \Re(f) = 0$, which are square integrable whenever $r (\varphi'(r))^2 \in L^1([0, r_0])$. The next step is then to compute $\partial_r \Im(f)$ and $\partial_\theta \Im(f)$. We use the shorthands 
\begin{align*}
	\tau &= \theta + 2\pi, &  u &= (2r - g(\tau))^{-1}.
\end{align*}
For $\partial_r \Im(f)$, we have $\partial_r u = -2u^2$, and hence
\[
	\partial_r \left( -\log W (u) \right)
	= -\frac{1}{W(u)} \frac{W(u)}{u(1 + W(u))} \left( - 2u^2  \right)
	= \frac{2u}{1 + W(u)}.
\]
For $\partial_\theta \Im(f)$, we have $\partial_\theta u = -u^2 \cdot (-g'(\tau)) = - u^2(1 + \log(\tau))/(\tau^2 \log^2 \tau)$, and hence
\[
	\partial_\theta \left( -\log W (u) \right) = \frac{(1 + \log \tau) u}{(1 + W(u))\tau^2 \log^2 \tau}.
\]
We hence arrive at the derivative matrix
\[
\begin{bmatrix}
	\partial_r \Re(f) & r^{-1} \partial_\theta \Re(f)\\
	\partial_r \Im(f) & r^{-1} \partial_\theta \Im(f)
\end{bmatrix}
=
\begin{bmatrix}
	\varphi'(r) & 0\\
	\dfrac{2u}{1 + W(u)} & \dfrac{(1 + \log \tau) u}{r (1 + W(u))\tau^2 \log^2 \tau}
\end{bmatrix}.
\]

To estimate these derivatives, we note that $g(\tau) \leq 2r - g(\tau) \leq g(\theta)$ in our region. Inverting all terms, it follows that $\theta \log \theta \leq u \leq \tau \log \tau$. Since $W$ is increasing and $W(t \log t) = \log t$, we hence have $\log \theta \leq W(u) \leq \log \tau$. Moreover, recalling the notation from the beginning of Section \ref{subsect:cusp_proofs}, it is reasonably easy to see that $\log(\theta) \approx \log(\tau)$ and $g(\theta) \approx g(\tau)$ for $\theta \in [\theta_0, \infty)$. In particular, we have
\begin{align*}
	u &\approx \theta \log \theta && \text{and}&
	W(u) &\approx \log \theta.
\end{align*}

We can then estimate $\abs{\partial_r \Im f(r, \theta)}$ from both sides by
\[
	\frac{2\theta \log \theta}{1 + \log \tau}
	\leq \abs{\partial_r \Im f(r, \theta)} 
	\leq \frac{2\tau \log \tau}{1 + \log \theta},
\]
implying that
\begin{equation}\label{eq:partial_r_estimate}
	\abs{\partial_r \Im f(r, \theta)}
	\approx \theta.
\end{equation}
To bound $r^{-1} \partial_\theta \Im(f)$, we first use the above estimates to obtain
\begin{equation*}
	\frac{1}{\theta \log \theta} \lesssim 
	\frac{\theta \log \theta}{\tau^2 \log^2 \tau}
	\leq \partial_\theta \Im f(r, \theta)
	\leq \frac{1 + \log \tau}{\tau \log \tau (1 + \log \theta)}
	\lesssim \frac{1}{ \theta \log \theta}.
\end{equation*}
That is, $\partial_\theta \Im f(r, \theta) \approx g(\theta)$. Then, since $g(\tau) \leq r < h(\theta) \leq g(\theta)$ in our domain, and since $g(\tau) \approx g(\theta)$, we in fact have $r \approx g(\theta)$. Hence,
\begin{equation}\label{eq:partial_theta_estimate}
	\frac{\partial_\theta \Im f(r, \theta)}{r} \approx 1.
\end{equation}
In particular, the function $r^{-1} \partial_\theta \Im(f)$ is bounded and hence clearly square integrable over $A$. 

Next, we check the square integrability of $\partial_r \Im(f)$. We begin by investigating the integral of an arbitrary function of $\theta$ over $A$ under our chosen parametrization. Letting $F \colon [\theta_0, \infty) \to [0, \infty)$, we use polar integration to get
\[
	\int_{A} F(\theta) = \int_{\theta_0}^\infty \int_{g(\tau)}^{h(\theta)} F(\theta) r \dd r \dd \theta
	\leq \int_{\theta_0}^\infty (h(\theta) - g(\tau)) \frac{F(\theta)}{\theta \log \theta} \dd \theta.
\]
Moreover, we have
\begin{multline*}
	h(\theta) - g(\tau) = \frac{g(\theta) + g(\theta + 2\pi)}{2} - g(\theta + 2\pi) = \frac{g(\theta) - g(\theta + 2\pi)}{2}\\
	= \frac{1}{2} \left( \frac{1}{ \theta \log \theta} - \frac{1}{(\theta + 2\pi) \log (\theta + 2\pi)} \right)
	= \frac{(\theta + 2\pi) \log (\theta + 2\pi) - \theta \log \theta}{2 \theta (\theta + 2\pi) \log \theta \log (\theta + 2\pi)}\\
	= \frac{2\pi \log (\theta + 2\pi) + \theta \log (1 + 2\pi/\theta)}{2 \theta (\theta + 2\pi) \log \theta \log (\theta + 2\pi)} \lesssim \frac{1}{\theta^2 \log \theta}.
\end{multline*}
Note in particular that in the last step of the above computation, we have $\theta \log(1 + 2\pi/\theta) = \log((1 + 2\pi/\theta)^\theta) \to \log \exp(2\pi) = 2\pi$ as $\theta \to \infty$, so hence the dominant term in the numerator is $2\pi \log (\theta + 2\pi)$. We thus finish our estimate as follows:
\begin{equation}\label{eq:polar_computation}
	\int_{A} F(\theta)
	\leq \int_{\theta_0}^\infty (h(\theta) - g(\tau)) \frac{F(\theta)}{\theta \log \theta} \dd \theta
	\lesssim \int_{\theta_0}^\infty \frac{F(\theta)}{\theta^3 \log^2 \theta}.
\end{equation}
Now, since $\abs{\partial_r \Im(f)}^2 \lesssim \theta^2$ by \eqref{eq:partial_r_estimate}, we conclude that $\abs{\partial_r \Im(f)} \in L^2(A)$ by taking $F(\theta) = \theta^2$ in \eqref{eq:polar_computation}, and observing that the resulting integrand $\theta^{-1} \log^{-2} \theta$ has a finite integral. We thus conclude that if $r(\varphi'(r))^2 \in L^1([0, r_0])$, then $f \in W^{1,2}(\Omega \setminus \{0\}, \mathbb{C})$, and consequently $f \in W^{1,2}(\Omega, \mathbb{C})$ by removability of isolated points for planar $W^{1,2}$-spaces.

It remains to find suitable choices of $K$ and $\Sigma$. We choose $\Sigma \equiv 0$ in this region, in which case we require $K \geq \abs{Df}^2/J_f$. By \eqref{eq:partial_r_estimate} and \eqref{eq:partial_theta_estimate}, we have
\[
	\abs{Df(r, \theta)}^2 \lesssim (\varphi'(r))^2 + \theta^2. 
\]
On the other hand, we have
\[
	J_f(r, \theta) \gtrsim \varphi'(r).
\]
Consequently, we may choose $\Sigma$ and $K$ so that
\begin{equation}\label{eq:spiral_second_region_choices}
	\Sigma(r, \theta) = 0, \qquad \qquad K(r, \theta) \approx \varphi'(r) + \frac{\theta^2}{\varphi'(r)} + 1. 
\end{equation}

\subsection{The results}

It remains now to state our choices of $\varphi$ and the resulting counterexamples. We begin with Theorem \ref{thm:discontinuity_bounded_sigma}, recalling first its statement.

\begin{customthm}{\ref{thm:discontinuity_bounded_sigma}} 
	There exist a domain $\Omega \subset \mathbb R^2$ and a Sobolev map $f\in W^{1,2} (\Omega , \mathbb R^2)$ such that $0\in \Omega$, $f \in C (\Omega \setminus \{0\}, \mathbb R^2)$, $\lim_{x \to \infty} \abs{f(x)}=\infty$, and $Df \in \mathcal M_2(K, \Sigma)$ with
	\[
		\Sigma \in L^\infty (\Omega ) 
		\qquad \text{and} \qquad K\in L^{1}(\Omega)\,.  
	\]
\end{customthm}
\begin{proof}
	We use the above construction with $\theta_0 = 2\pi$, $r_0 = 1$, and
	\[
		\varphi(r) = r.
	\]
	Notably, $\varphi$ is Lipschitz, and consequently the resulting map $f$ is Lipschitz in $B$. This choice indeed satisfies $r (\varphi'(r))^2 = r \in L^1([0,1])$, so $\abs{Df} \in L^2(\Omega \setminus \{0\})$. Moreover, by \eqref{eq:spiral_first_region}, both $\Sigma$ and $K$ are constant in the region $B$. Since $\Sigma \equiv 0$ in the other region $A$, we have $\Sigma \in L^\infty(\Omega)$. For $K \in L^1(\Omega)$, since $\varphi'$ is constant, it suffices by  \eqref{eq:spiral_second_region_choices} to show that
	\[
		\int_{A} \theta^2 < \infty.
	\]
	But this is true by \eqref{eq:polar_computation} with yet again $F(\theta) = \theta^2$. Finally, as $x \to 0$, the imaginary part of $f(x)$ clearly tends to infinity.
\end{proof}

The remaining result to prove is Theorem \ref{thm:spiral_ex_version_2}.

\begin{proof}[Proof of Theorem~\ref{thm:spiral_ex_version_2}]
	The case $q = \infty$ is exactly the result of Theorem \ref{thm:discontinuity_bounded_sigma}. Hence, we may assume that $q \in [1, \infty)$.
	
	We use the above construction, this time with the choice
	\[
		\varphi(r) = \int_0^r t^{2p^{-1} - 2} \log^{-7/4 + p^{-1}} t^{-1} \dd t.
	\]
	Note that by $p \in [1, 2]$, we have $2p^{-1} - 2 \geq -1$. Moreover, the case $2p^{-1} - 2 = -1$ corresponds to $p = 2$, in which case $-7/4 + p^{-1} = -5/4 < -1$. Hence, the integral used to define $\varphi(r)$ is finite for all $r > 0$ small enough, and we may hence choose $r_0$ and $\theta_0$ so that $\varphi(r)$ is a finite-valued increasing function on $[0, r_0]$. By our choice of $\varphi$, we have
	\begin{equation}\label{eq:gen_spiral_phi_choice}
		\varphi'(r) = r^{2p^{-1} - 2} \log^{-7/4 + p^{-1}} r^{-1}.
	\end{equation}
	
	We first determine the degree of integrability of $\varphi'(\abs{x})$ over $\Omega$, as this is used for many parts in the verification that our example is as desired. Indeed, if $s \in [1, \infty)$, we have by \eqref{eq:gen_spiral_phi_choice} that
	\begin{equation}\label{eq:phi_deriv_degree_of_int}
		\int_{\Omega} \left( \varphi'(\abs{x}) \right)^s \lesssim \int_0^{r_0} \frac{\dd r}{r^{2s - 2sp^{-1} - 1} \log^{7s/4 - sp^{-1}} r^{-1}} 
	\end{equation}
	This integral is finite if $2s - 2sp^{-1} - 1 \leq 1$, which is equivalent to $p^{-1} + s^{-1} \geq 1$. Note that in the extremal case $p^{-1} + s^{-1} = 1$, the finiteness of the integral also requires that $7s/4 - sp^{-1} > 1$; however, this condition rearranges to $p^{-1} + s^{-1} < 7/4$, which holds in the extremal case since $p^{-1} + s^{-1} = 1$.
	
	We have $\Sigma(x) \leq 6 + 3(\varphi'(\abs{x}))^2$ and $\Sigma(x)/K(x) \leq 6 + 3 \varphi'(\abs{x})$ in $B$ by \eqref{eq:spiral_first_region}, and we also have $\Sigma = \Sigma/K \equiv 0$ in $A$. Hence, \eqref{eq:phi_deriv_degree_of_int} with $s = q$ yields that $\Sigma/K \in L^q(\Omega)$ and $\Sigma \in L^{q/2}(\Omega)$ if $p^{-1} + q^{-1} \geq 1$. Moreover, $\abs{Df} \in L^2(\Omega)$ was shown to be equivalent with $r (\varphi'(r))^2 \in L^1([0, r_0])$: referring to \eqref{eq:phi_deriv_degree_of_int} with $s = 2$, this is true if $p^{-1} + 2^{-1} \geq 1$, which holds due to our assumption that $p \leq 2$. As our last application of \eqref{eq:phi_deriv_degree_of_int}, we have by \eqref{eq:spiral_first_region} that $K \in L^p(B)$ if $\varphi'(r) \in L^p(B)$: this is true if $2p \leq 4$, which again holds by our assumption that $p \leq 2$. 
	
	It remains to show that $K \in L^p(A)$. For this, it suffices by \eqref{eq:spiral_second_region_choices} to show the $L^p$-integrability of $\varphi'(r)$ and $\theta^2/\varphi'(r)$ over $A$. Since $K(r, \theta) \geq \varphi'(r)$ in $B$, the $\varphi'(r)$-term is covered by the same argument as used previously for $K \in L^p(B)$. For the other term, we again use \eqref{eq:polar_computation}. Indeed, we have
	\begin{multline*}
		\left(\frac{\theta^2}{\varphi'(r)}\right)^p = \frac{\theta^{2p} r^{2p - 2}}{\log^{-7p/4 + 1} r^{-1}}\\
		\leq \frac{\theta^{2p} (\theta \log \theta)^{2 - 2p}}{ \log^{-7p/4 + 1} ((\theta + 2\pi)\log(\theta + 2\pi))}
		\lesssim \frac{\theta^2}{\log^{p/4 - 1}(\theta)}
	\end{multline*}
	for all $\theta \in [\theta_0, \infty)$. Selecting $F(\theta) = \theta^2 \log^{1 - p/4} (\theta)$, the resulting integrand $\theta^{-1} \log^{-1 - p/4}(\theta)$ in \eqref{eq:polar_computation} is integrable whenever $-1 - p/4 < -1$, which is clearly true. We conclude that $K \in L^p(\Omega)$, completing the proof.
\end{proof}


\bibliographystyle{abbrv}
\bibliography{sources}

\end{document}